\documentclass{article}

\usepackage[utf8]{inputenc}

\usepackage{amsmath}
\usepackage{amsthm}
\usepackage{amssymb}
\usepackage[all]{xy}
\usepackage{dsfont}

\usepackage{mathdots}

\usepackage{tikz}
\usetikzlibrary{cd, arrows.meta}

\usepackage{enumerate}
\usepackage{hyperref}

\newcommand{\F}{\mathbb{F}}
\newcommand{\Z}{\mathbb{Z}}

\newcommand{\E}{\mathbb{E}}

\newcommand{\cC}{\mathcal{C}}
\newcommand{\cS}{\mathcal{S}}
\newcommand{\xto}{\xrightarrow}
\newcommand{\one}{\mathds{1}}
\newcommand{\bS}{\mathbb{S}}

\DeclareMathOperator{\Tor}{Tor}
\DeclareMathOperator{\HKR}{HKR}
\DeclareMathOperator{\THH}{THH}

\DeclareMathOperator{\Mod}{Mod}

\DeclareMathOperator{\Alg}{Alg}

\DeclareMathOperator{\Free}{Free}
\DeclareMathOperator{\aug}{aug}

\DeclareMathOperator{\pt}{pt}
\DeclareMathOperator{\Ba}{Bar}
\DeclareMathOperator{\GL}{GL}
\DeclareMathOperator{\BGL}{BGL}
\DeclareMathOperator{\Sp}{Sp}
\DeclareMathOperator{\HH}{HH}

\DeclareMathOperator{\gp}{gp}
\DeclareMathOperator{\dlog}{dlog}
\newtheorem{thm}{Theorem}[section]
\newtheorem{lem}[thm]{Lemma}
\newtheorem{prop}[thm]{Proposition}
\newtheorem{defn}[thm]{Definition}

\theoremstyle{definition}

\newtheorem{example}[thm]{Example}
\newtheorem{remark}[thm]{Remark}

\newcommand{\pow}[1]{[\kern-0.1em[#1]\kern-0.1em]}

\title{B\"okstedt periodicity and quotients of DVRs}
\author{Achim Krause, Thomas Nikolaus}

\begin{document}
\maketitle

\begin{abstract}
In this note we compute the topological Hochschild homology of quotients of DVRs. Along the way we give a short argument for B\"okstedt periodicity and generalizations over various other bases. Our strategy also gives a very efficient way to redo the computations of  $\THH$ (resp. logarithmic $\THH$) of complete DVRs originally due to Lindenstrauss-Madsen (resp. Hesselholt--Madsen).
\end{abstract}

\section*{Introduction}

Topological Hochschild homology, THH, together with its induced variant topological cyclic homology, TC, has been one of the major tools to compute algebraic $K$-theory in recent years. It also is an important invariant in its own right, due to its connection to $p$-adic Hodge theory and crystalline cohomology \cite{MR3905467,MR3949030}.

The key point is that topological Hochschild homology $\THH_*(R)$, as opposed to algebraic $K$-theory, can be completely identified for many rings $R$. Let us list some examples here.
\begin{enumerate}
\item
The most fundamental result in the field is B\"okstedt periodicity, which states that 
$
\THH_*(\F_p) = \F_p[x]
$
for a class $x$ in degree $2$. This is also the input for the work of Bhatt--Morrow--Scholze \cite{MR3949030}.
\item
The $p$-adic computation of $\THH_*(\Z_p)$ was also done by B\"okstedt and eventually lead to the $p$-adic identification of $K_*(\Z_p)$, see \cite{MR1317117, MR1663391}.
\item
More generally, Lindenstrauss and Madsen identify $\THH_*(A)$ $p$-adically for a complete DVR $A$ with perfect residue field $k$ of characteristic $p$ \cite{MR1707702}. This computation was one of the key inputs for Hesselholt and Madsen's seminal computation of K-theory of rings of integers in $p$-adic number fields.  
\item Brun computed $\THH_*(\Z/p^n)$ in \cite{MR1750729}. This gives some information about $K_*(\Z/p^n)$, which is still largely unknown, see \cite{MR1823499}.
\end{enumerate}

In this paper we revisit all the $\THH$-computations mentioned above from scratch, and give new, easier and more conceptual proofs. We will go one step further and give a complete formula for 
$\THH_*(A')$ where $A' = A/\pi^k$ is a quotient of a DVR $A$ with perfect residue field of characteristic $p$. We identify $\THH_*(A')$ with the homology of an explicitly described DGA, see Theorem \ref{thm_quot}.  
This for example recovers the computation of $\THH_*(\Z/p^k)$ by Brun and also identifies the ring structure in this case (which was unknown so far). The result shows an interesting dichotomy depending on how large $k$ is when compared to the $p$-adic valuation of the derivative of the  minimal polynomial of a uniformizer of $A$
 (relative to the  Witt vectors of the residue field), see Section \ref{eval}.

The main new idea employed in this paper  is to first compute $\THH$ of $A$ and $A/\pi^k$ relative to the spherical group ring $\bS[z]$. This relative $\THH$ of $A$ satisfies a form of B\"okstedt periodicity, which was to our knowledge first observed by J. Lurie, P. Scholze and B.Bhatt. It appeared in work of Bhatt-Morrow-Scholze \cite{MR3949030} as well as in \cite{MR3874698}. But the maneuver of working relative to the uniformizer is much older in the algebraic context, for example in the theory of Breuil-Kisin modules \cite{MR1293975, MR1669039,MR2600871}.\footnote{We would like to thank Matthew Morrow and Lars Hesselholt for pointing this out and explaining the history to us.} 

Finally, having computed THH relative to $\bS[z]$ we use a descent style spectral sequence (see Section \ref{section5} and Section \ref{sec:thhquot}) to recover the absolute $\THH$. 
In Section \ref{logTHH} we also deduce the computation of logarithmic $\THH$ of CDVRs (due to Hesselholt--Madsen) from the computation of relative $\THH$ using a similar spectral sequence. 

\setcounter{tocdepth}{1}
\renewcommand{\baselinestretch}{0.75}\normalsize
\tableofcontents
\renewcommand{\baselinestretch}{1.0}\normalsize
\subsection*{Conventions}
We freely use the language of $\infty$-categories and spectra. The sphere spectrum is denoted by $\bS$. For a commutative ring $R$ there is an associated commutative ring spectrum which we abusively also denote by $R$. In this situation we have the ring spectra $\HH(R)$ (`Hochschild homology') and $\THH(R)$ (`Topological Hochschild homology') defined as
\[
\HH(R / \Z) = R \otimes_{R \otimes_\Z R} R \qquad \THH(R) =  R \otimes_{R \otimes_\bS R} R \ .
\]
We denote the homotopy groups of these spectra by $\THH_*(R)$ and $\HH_*(R)$. 
More generally there are relative versions for a ring $R$ over a base ring (spectrum) $S$ given as 
$\THH(R / S) = R \otimes_{R \otimes_{S} R} R$ and similar for $\HH$. Note that Hochschild homology as defined here is equivalent to $\THH(R/ \Z)$ and is automatically fully derived. It thus agrees with what is classically called Shukla homology.  

We shall denote the $p$-completion of the spectrum $\THH(R)$ by $\THH(R;\Z_p)$ and the homotopy groups accordingly by $\THH_*(R;\Z_p)$. Note that these are in general not the $p$-completions of the groups $\THH_*(R)$, but in the case that the groups $\THH_*(R)$ have bounded order of $p^\infty$-torsion this is true. There is the commonly used conflicting notation $\THH_*(R;R')$ for THH with coefficients in an $R$-algebra $R'$, given by he homotopy groups of $\THH(R) \otimes_R R'$. To avoid confusion we do not use the notation $\THH(R;R')$ in this paper.

Finally, there are useful equivalences
\begin{align*}
&\THH(A \otimes_\bS B) = \THH(A) \otimes_{\bS} \THH(B) \\
&\THH(R/S) = \THH(R) \otimes_{\THH(S)} S \\
&\THH(A) \otimes_{\bS} B = \THH(A \otimes_{\bS} B / B)
\end{align*}
and some variants which are straighforward to prove and will be used frequently.

\subsection*{Acknowledgments}
We would like to thank Lars Hesselholt, Eva H\"oning, Mike Mandell, Matthew Morrow, Peter Scholze and Guozhen Wang
 for helpful conversations. We also thank Lars Hesselholt and Eva H\"oning for comments on a draft.
 The authors were funded by the Deutsche Forschungsgemeinschaft (DFG, German Research Foundation) under Germany's Excellence Strategy EXC 2044 390685587, Mathematics M\"unster: Dynamic--Geometry--Structure.

%\vspace{1cm}

\section{B\"okstedt periodicity for $\F_p$}\label{sec_bokstedt}

We want to give a proof of the fundamental result of B\"okstedt, that topological Hochschild homology of $\F_p$ is a polynomial ring on a degree 2 generator. The proof presented here is closely related to the Thom spectrum proof in \cite{Blumberg} based on a result of Hopkins-Mahowald, but in our opinion it is more direct, see Appendix \ref{Thom} for a precise discussion.

Let us first give a slightly more conceptual formulation of B\"okstedt's result.
\begin{thm}[B\"okstedt]\label{thm:bok}
The spectrum $\THH(\F_p)$ is as an $\E_1$-algebra spectrum over $\F_p$
 free on a generator $x$ in degree 2, i.e. equivalent to $\F_p[\Omega S^3]$. 
\end{thm}

Here $\F_p[\Omega S^3]$ is the group ring of the $\E_1$-group $\Omega S^3$ over $\F_p$ i.e. the $\F_p$-homology $\F_p \otimes_\bS \Sigma^\infty_+ \Omega S^3$. The equivalence between the two formulations relies on the fact that $\Omega S^3$ is the free $\E_1$-group on $S^2$, where $S^2$ is considered as a pointed space.  

Our proof relies on a structural result about the dual Steenrod algebra $\F_p \otimes_\bS \F_p$. We consider this spectrum as an $\F_p$-algebra using the inclusion into the left factor.\footnote{If we use the right factor this produces an equivalent $\F_p$-algebra where the equivalence is the conjugation.} It is an $\mathbb{E}_\infty$-algebra over $\F_p$, but has a universal description as an $\mathbb{E}_2$-algebra. This result seems to be known, at least to some experts, but we have not been able to find it written up in the literature.

\begin{thm}\label{thm:free}
As an $\E_2$-$\F_p$-algebra, the spectrum $\F_p \otimes_\bS \F_p$ is free on a single generator of degree $1$, i.e. it is as an $\E_2$-$\F_p$-algebra equivalent to $\F_p [\Omega^2 S^3]$.
\end{thm}

We will give a proof of Theorem \ref{thm:free} in the next section. But let us first deduce Theorem \ref{thm:bok} from it. 
\begin{proof}[Proof of Theorem \ref{thm:bok}]
We have an equivalence of $\E_1$-algebras
\begin{align*}
\THH(\F_p)& \simeq
\F_p \otimes_{\F_p \otimes_\bS \F_p} \F_p \\
&\simeq \F_p \otimes_{\F_p[\Omega^2 S^3]} \F_p  \\
&\simeq \F_p\left[  \Ba(\pt, \Omega^2 S^3, \pt) \right] \\
&\simeq \F_p [ \Omega S^3] .
\end{align*}
The third equivalence uses that $\F_p[-]$ sends products to tensor products and preserves colimits.
\end{proof}

\begin{remark}
If one only wants to use that $\F_p \otimes_\bS \F_p$ is free as an abstract $\E_2$-algebra and avoid space level arguments, one can observe that in any pointed presentably symmetric monoidal $\infty$-category $\cC$ one has for every object $X \in C$ an equivalence
\[
\one \otimes_{\Free_{\E_{n+1}(X)}} \one \simeq \Free_{\E_n}(\Sigma X) \ .
\]
This is proven in \cite[Corollary 5.2.2.13]{HA} for $n=0$ and the case  $n >0$ can be reduced to this case using Dunn Additivity by replacing $\cC$ with the $\infty$-category of augmented $\E_n$-algebras $\Alg_{\E_n}^{\aug}(\cC)$. This $\infty$-category satisfies the assumptions of \cite[Corollary 5.2.2.13]{HA} by \cite[Proposition 5.1.2.9]{HA}.
\end{remark}

\subsection{Proof of Theorem \ref{thm:free}}

In order to prove this result we first recall that for every $\E_2$-ring spectrum $R$ over $\F_2$ there exist Dyer-Lashof operations 
$$
Q^i: \pi_k R \to \pi_{k+i}R
$$
for $i  \leq  k+1$ and they satisfy all the relations of the usual Dyer-Lashof operations as long as they make sense. 
For an $\E_2$-algebra $R$ over $\F_p$ with odd $p$, there exist operations 
\begin{align*}
Q^i:& \pi_kR \to \pi_{k+2i(p-1)}R\\
\beta Q^i :& \pi_k R \to \pi_{k+2i(p-1)-1}R
\end{align*}
for $i \leq 2k +1$. %Then the following classical computation is the key to the proof of Theorem \ref{thm:free}.

\begin{prop}\label{prop:key}
Let $R$ be the free $\E_2$-algebra over $\F_p$ on a generator in degree 1. Then
\begin{enumerate}
\item
for $p=2$  we have 
$$\pi_*R =\F_2[x_1, x_2, \ldots]$$ where $|x_i| = 2^i -1$.
The element $x_{i+1}$ is given by 
$Q^{2^i}Q^{2^{i-1}} \ldots Q^8Q^{4}Q^2 x_1$. In addition, $\beta x_i = x_{i-1}^2$.
\item
for $p$ odd we have
$$\pi_*R =\Lambda_{\F_p}(y_0, y_1, \ldots) \otimes \F_p[z_1, z_2, \ldots]$$ where $|y_i| = 2p^i - 1$, $|z_i|= 2p^i - 2$.
The element $y_{i+1}$ is given by $Q^{p^i}\ldots Q^p Q^1 y_0$, the element $z_i$ is given by $\beta Q^{p^i}\ldots Q^p Q^1 y_0$.
\end{enumerate}
Any $\E_2$-algebra $R$ over $\F_p$ whose homotopy ring together with the action of the Dyer-Lashof operations is of the above form, is also free on a generator in degree 1. 
\end{prop}
\begin{proof}
We use that $R \simeq \F_p \otimes \Omega^2 S^3$, i.e. we are computing the Pontryagin ring of the space $\Omega^2 S^3$. Then 
the first part is due to Araki and Kudo \cite[Theorem 7.1]{MR0087948}, the second part is due to Dyer-Lashof \cite[Theorem 5.2]{MR0141112}. These results are relatively straightforward computations using the Serre spectral sequence and the Kudo transgression theorem. 

Now for the last part assume that we have given any such $R$ and any non-trivial element $x_1 \in \pi_1(R)$. We get an induced map from the free algebra $\Free_{\E_2}(x_1) \to R$. Since this map is an $\E_2$-map the induced map on homotopy groups is compatible with the ring structure as well as the Dyer-Lashof operations. But everything is generated from $x_1$ under these operations in the same way, so the map is an equivalence.
\end{proof}

\begin{proof}[Proof of Theorem \ref{thm:free}]
By Proposition \ref{prop:key} we only have to verify that the homotopy groups $\F_p \otimes \F_p$ have the correct ring structure and Dyer-Lashof operations. This is a classical calculation due to Milnor for the ring structure and Steinberger \cite[Chapter 3, Theorem 2.2 and 2.3]{MR836132} for the Dyer-Lashof operations: at $p=2$, the generator $x_i$ corresponds to the Milnor basis element $\overline{\zeta}_i$, at $p$ odd $z_i$ corresponds to the element $\overline{\xi}_i$ and $y_i$ to $\overline{\tau}_i$. \end{proof}

\begin{remark}\label{rem_equib}
We want to remark that Theorem \ref{thm:bok} also implies Theorem \ref{thm:free}.
Thus assume that Theorem \ref{thm:bok} holds. 
We have that $\pi_1(\F_p \otimes_{\bS} \F_p)$ is isomorphic to $\F_p$, generated by an element $b$. We can thus choose an $\E_2$-map
\[
\Free_{\E_2}(b) \to \F_p\otimes_\bS \F_p
\]
which induces an equivalence on $1$-types.\footnote{The computation of the first two homotopy groups of $\F_p\otimes_\bS \F_p$ is everything that we input about the dual Steenrod algebra. So in fact even Milnor's computation, as well as the results of Steinberger cited here, could be recovered from an independent proof of B\"okstedt's result.} We can form the Bar construction on these augmented $\F_p$-algebras, and the resulting map
\[
\Free_{\E_1}(x) \to \THH(\F_p)
\]
is an equivalence on $\pi_2$, so by Theorem \ref{thm:bok} it is an equivalence. Thus, Theorem \ref{thm:free} follows from the following lemma.
\end{remark}

\begin{lem}
Let $A\to B$ be a map augmented connected $\E_1$-algebras over $\F_p$. Then if the map
\[
\F_p\otimes_A \F_p \to \F_p\otimes_B \F_p
\]
is an equivalence, so is $A\to B$.
\end{lem}
\begin{proof}
Assume $A\to B$ is not an equivalence. Let $d$ denote the connectivity of the cofiber of $A\to B$, i.e. $\pi_i(B/A) = 0$ for $i< d$, but $\pi_d(B/A)\neq 0$. $\F_p\otimes_A \F_p$ admits a filtration (obtained by filtering the Bar construction over $\F_p$ by its skeleta) whose associated graded is given in degree $n$ by $\Sigma^n (A/\F_p)^{\otimes_{\F_p} n}$. Here $A/\F_p$ is the cofiber of $\F_p\to A$ and $1$-connective by assumption. The map 
\begin{equation*}
\Sigma^n (A/\F_p)^{\otimes_{\F_p} n} \to \Sigma^n (B/\F_p)^{\otimes_{\F_p} n}
\end{equation*}
has $(d + 2n - 1)$-connective cofiber. Thus, the $(d+1)$-type of the cofiber of $\F_p\otimes_A \F_p \to \F_p\otimes_B \F_p$ receives no contribution from the terms for $n\geq 2$, and coincides with the $(d+1)$-type of the cofiber of $\Sigma (A/\F_p) \to \Sigma (B/\F_p)$, which is $\Sigma (B/A)$ and has nonvanishing $\pi_{d+1}$ by assumption. So $\F_p\otimes_A \F_p \to \F_p\otimes_B \F_p$ cannot have been an equivalence.
\end{proof}

%%%%%%%%%%%%%%%%%%%%%%%%%%%%%%%
\section{B\"okstedt periodicity for perfect rings}\label{sec_2}
%%%%%%%%%%%%%%%%%%%%%%%%%%%%%%%

Now we also want to recover the well-known calculation of $\THH$ for a perfect $\F_p$-algebra $k$. This can directly  be reduced to B\"okstedt's theorem. Let us first note that there is a morphism $\THH(\F_p) \to \THH(k)$ induced from the map $\F_p \to k$. Moreover the spectrum $\THH(k)$ is a $k$-module, so that we get an induced map
\begin{equation}\label{perfect}
k[x] \cong k \otimes_{\F_p}  \THH(\F_p) \to \THH(k)
\end{equation}
where the first term $k[x]$ denotes the free $\E_1$-algebra on a generator in degree 2.

\begin{prop}\label{prop_perfect}
For a perfect $\F_p$-algebra $k$ 
the map \eqref{perfect} is an equivalence. 
\end{prop}
\begin{proof}
Recall that for every perfect $\F_p$-algebra $k$ there is a $p$-complete $\E_\infty$-ring spectrum $\bS_{W(k)}$, called the spherical Witt vectors, with $\pi_0(\bS_{W(k)}) = W(k)$ and which is flat over $\bS_p$. It follows that the homology $\Z \otimes_{\bS} \bS_{W(k)}$ is given by $W(k)$ and thus the $\F_p$-homology $\F_p  \otimes_\bS \bS_{W(k)} $ by $k$. 

In particular we get that 
\begin{align*}
\THH(k) &= \THH(\F_p \otimes_\bS \bS_{W(k)}) \\
&=   \THH(\F_p)  \otimes_\bS \THH(\bS_{W(k)}) \\
&= \THH(\F_p)\otimes_{\F_p} (\F_p \otimes_\bS \THH(\bS_{W(k)})) \\
& = \THH(\F_p)\otimes_{\F_p} \HH(k / \F_p)  \ .
\end{align*}
where $\HH(k / \F_p)$ is the Hochschild homology of $k$ relative to $\F_p$. The result now follows once we know that this is given by $k$  concentrated in degree $0$. This immediately follows from the vanishing of the cotangent complex of $k$ but we want to give a slightly different argument here: 

It suffices to show that the positive dimensional groups $\HH_i(k/\F_p)$ are zero. To see this it is enough to show that for every $\F_p$-algebra $A$ the Frobenius $\varphi: A \to A$ induces the zero map $\HH_i(A/ \F_p) \to \HH_i(A / \F_p)$ for $i > 0$, since for $A = k$ perfect the Frobenius is also an isomorphism. Now for general $A$ this follows since $\HH(A/\F_p)$ is a simplicial commutative $\F_p$-algebra and the Frobenius $\varphi$ acts through the levelwise Frobenius. But the levelwise Frobenius for every simplicial commutative $\F_p$-algebra induces the zero map in positive dimensional homotopy. \footnote{This follows since for every simplicial commutative $\F_p$-algebra $R_\bullet$ the Frobenius can be factored as $\pi_n(R_\bullet) \to \pi_n(R_\bullet)^{\times p} \to \pi_n(R_\bullet)$ where the latter map is induced by the multiplication $R_\bullet^{\times p} \to R_\bullet$ considered as a map of underlying simplicial sets. For $n > 0$ it follows by an Eckmann-Hilton argument that the multiplication map $\pi_n(R_\bullet) \times \pi_n(R_\bullet) \to \pi_n(R_\bullet)$ is at the same time multilinear and linear, hence zero.}
\end{proof}

\begin{remark}\label{rem_padic}
Note that the proof in particular shows that  $\THH(\bS_{W(k)})$ is $p$-adically equivalent to $\bS_{W(k)}$ as this can be checked on $\F_p$-homology. We will also write $\THH(\bS_{W(k)};\Z_p)$ for the $p$-completion of $\THH(\bS_{W(k)})$ so that we have 
\[
\THH(\bS_{W(k)};\Z_p) \simeq \bS_{W(k)} \ .
\]
Integrally this is not quite the case, as one encounters contributions form the cotangent complex $L_{W(k)/\Z}$ which only vanishes after $p$-completion.
\end{remark}
We also note that one can also deduce Proposition \ref{prop_perfect} from a statement similar to Theorem \ref{thm:free} which we want to list for completeness.

\begin{prop}
For $k$ a perfect $\F_p$-algebra, we have
\[
k\otimes_{\bS_{W(k)}} k = \Free^k_{\E_2}(\Sigma k).
\]
i.e. the spectrum $k\otimes_{\bS_{W(k)}} k$ is as an $\mathbb{E}_2$-$k$-algebra free on a single generator in degree 1. 
\end{prop}
\begin{proof}
As $\bS_{W(k)} \otimes_{\bS} \F_p = k$, we have 
\[
k\otimes_{\bS_{W(k)}} k = k \otimes_{\bS} \F_p = k\otimes_{\F_p} (\F_p \otimes_{\bS} \F_p),
\]
so the statement follows from base-changing the statement over $\F_p$.
\end{proof}

%%%%%%%%%%%%%%%%%%%%%%%%%%%%%%%%
\section{B\"okstedt periodicity for CDVRs}\label{sec_three}
%%%%%%%%%%%%%%%%%%%%%%%%%%%%%%%%%

Now we want to turn our attention to complete discrete valuation rings, abbreviated as CDVRs. % with perfect residiue field of characteristic $p$. 
We will determine their absolute $\THH$ later, but for the moment we focus on an analogue of B\"okstedt's theorem which works relative to the $\E_\infty$-ring spectrum
\[
\bS[z] := \bS[\mathbb{N}] = \Sigma^\infty_+ \mathbb{N}  \ .
\]
For a CDVR $A$ we let $\pi$ be a uniformizer, i.e. a generator of the maximal ideal, and consider it as a $\bS[z]$-algebra via $z\mapsto \pi$. Everything that follows will implicitly depend on such a choice. By assumption $A$ is complete with respect to $\pi$. 
Since $\pi$ is a non-zero divisor this is equivalent to being derived $\pi$-complete. Moreover $A$ if  has residue field of characteristic $p$ then $A$ is also (derived) $p$-complete since $p$ is contained in the maximal ideal. 

The following result is, at least in mixed characteristic, due to Bhargav Bhatt, Jacob Lurie and Peter Scholze  in private communication but versions of it also appear in  \cite{MR3949030} and in \cite{MR3874698}.%In 
\begin{thm}\label{thm:boekstedtDVR}
Let $A$ be a CDVR with perfect residue field of characteristic $p$.  Then we have 
\[
 \THH_*(A / \bS[z]; \Z_p)  = A[x]
\] 
for $x$ in degree 2. 
\end{thm}
\begin{proof}
We distinguish the cases of equal and of mixed characteristic. 
In mixed characteristic we have  the equation $(p) = (\pi^e)$ where $e$ is the ramification index. We deduce that  $\THH(A / \bS[z]; \Z_p)$ is $\pi$-complete since it is $p$-complete. 
Now we have
\[
\THH(A/ \bS[z];\Z_p)\otimes_A k = \THH(A / \bS[z]; \Z_p) \otimes_{\bS[z]} \bS = \THH(k)
\]
and thus this is by Proposition \ref{prop_perfect} given by an even dimensional polynomial ring over $k$. Thus $\THH(A/ \bS[z];\Z_p)$ is $\pi$-torsion free and the result follows. 

If $A$ is of  equal characteristic $p$ then $A$ is isomorphic to the formal power series ring $k\pow{z}$ where $k$ is the residue field (which is perfect by assumption). We consider the $\E_\infty$-ring $\bS_{W(k)}\pow{z}$ obtained as the $z$-completion of $\bS_{W(k)}[z]$. Then we have an equivalence
\[
k\pow{z} \simeq \F_p \otimes_\bS \bS_{W(k)}\pow{z}
\]
which uses that $\F_p$ is of finite type over the sphere. As a result, we get an equivalence
\begin{align*}
\THH(k\pow{z} / \bS[z]) & \simeq \THH(\F_p) \otimes_\bS \THH(  \bS_{W(k)}\pow{z} / \bS[z]) \\
& \simeq \THH(\F_p) \otimes_{\F_p} ( \F_p \otimes_\bS \THH(  \bS_{W(k)}\pow{z} / \bS[z]) ) \\
& \simeq \THH(\F_p) \otimes_{\F_p} \HH( k\pow{z} / \F_p[z]) \ .
\end{align*}
Now in order to show the claim it suffices to show that $\HH( k\pow{z} / \F_p[z])$ is concentrated in degree 0 (where it is given by $k\pow{z}$). In order to prove this we first note that $\F_p[z] \to k\pow{z}$ is (derived) relatively perfect, i.e. the square
\begin{equation}\label{pushout}
\xymatrix{
\F_p[z] \ar[r] \ar[d]^\varphi & k\pow{z}\ar[d]^\varphi \\
\F_p[z] \ar[r] & k\pow{z}
}
\end{equation}
is a pushout of commutative ring spectra, where $\varphi$ is the Frobenius. This holds because $1,z,...,z^{p-1}$ is basis for $\F_p[z]$ as a $\varphi(\F_p[z]) = \F_p[z^p]$-module and also for $k\pow{z}$ as a $\varphi(k\pow{z}) = k\pow{z^p}$-algebra. Now the map
\[
\pi_i(\HH(k\pow{z} / \F_p[z]) \otimes_{\F_p[z]} \F_p[z]) \to \pi_i\HH(k\pow{z} / \F_p[z])
\]
induced from the square \eqref{pushout} is an equivalence since the square is a pushout. We claim again, as in the proof of Proposition \ref{prop_perfect}, that this map is zero for $i>0$. Since $\varphi: \F_p[z] \to \F_p[z]$ is flat, we have
\[
\pi_i\HH(k\pow{z} / \F_p[z]) \otimes_{\F_p[z]} \F_p[z]) = \pi_i \HH(k\pow{z} / \F_p[z]) \otimes_{\F_p[z]} \F_p 
\]
as right $\F_p[z]$-modules. The map
\[
\pi_i\HH(k\pow{z} / \F_p[z]) \otimes_{\F_p[z]} \F_p[z] \to \pi_i\HH(k\pow{z} / \F_p[z])
\]
is induced up from the map $\pi_i\HH(k\pow{z}/\F_p[z]) \to \pi_i\HH(k\pow{z}/\F_p[z])$ induced by the Frobenius of $k\pow{z}$, which is given by the Frobenius of the simplicial commutative ring $\HH(k\pow{z} / \F_p[z])$. Thus it is zero on positive dimensional homotopy groups.
 \end{proof}
\begin{remark}
The isomorphism $ \THH_*(A / \bS[z]; \Z_p)  \cong A[x]$ of Theorem \ref{thm:boekstedtDVR} depends on the choice of generator $x$ of $\THH_2(A / \bS[z]; \Z_p)$. The proof of Theorem \ref{thm:boekstedtDVR} determines $x$ in mixed characteristic only  modulo $\pi$. We will see later that there is in fact a preferred choice of generator $x$ which then makes the isomorphism of Theorem \ref{thm:boekstedtDVR} canonical, see  Remark \ref{rem_generator}.
\end{remark}
\begin{remark}
Let $A$ be a not necessarily complete DVR of mixed characteristic $(0,p)$ with perfect residue field. Then we have that
\[
\THH(A/\bS[z];\Z_p) \to \THH(A_p / \bS[z];\Z_p) 
\]
is an equivalence where $A_p$ is the $p$-completion of $A$. This is true for every ring $A$. But for a DVR the $p$-completion $A_p$ is the same as the completion of $A$ with respect to the maximal ideal so that Theorem \ref{thm:boekstedtDVR} applies to yield that 
\[
\THH_*(A/\bS[z];\Z_p) = A_p[x] \ .
\]
For every prime $\ell \neq p$ we have that 
\[
\THH(A / \bS[z] ; \Z_\ell) = 0
\]
since $\ell$ is invertible in $A$.
If we can show that $\THH(A / \bS[z])$ is finitely generated in each degree we can therefore even get that $\THH_*(A / \bS[z]) = A[x]$ without $p$-completion. For example if $A = \Z_{(p)}$ or more generally localizations of rings of integers at prime ideals. But in general one can not control the rational homotopy type of $\THH(A/\bS[z])$, as the example of $\Z_p$ shows, where we get contributions from $\Z_p \otimes_\Z \Z_p$.

In equal characteristic we do not know how to compute $\THH_*(A/\bS[z];\Z_p)$ if $A$ is not complete, since in general the cotangent complex $L_{A / \F_p[z]}$ does not vanish.\footnote{For an explicit counterexample consider an element $f$ in the fraction field $Q(\F_p\pow{z})$ which is transcendental over $Q(\F_p[z])$. This exists for cardinality reasons. Now the cotangent complex $L_{Q(\F_p[z])(f) / Q(\F_p[z])}$ is nontrivial. Since it agrees with a localisation of $L_{A/\F_p[z]}$, where $A = \F_p\pow{z} \cap Q(\F_p[z])(f)$, $A$ is a DVR with nontrivial $L_{A/\F_p[z]}$.}
\end{remark}

\begin{remark}
One can also deduce the mixed characteristic version of  Theorem \ref{thm:boekstedtDVR} from an analogue of Theorem \ref{thm:free} which under the same assumptions as Theorem \ref{thm:boekstedtDVR} and in mixed characteristic states that $A\otimes_{\bS_{W(k)}[z]} A$ is $p$-adically the free $\E_2$-algebra on a single generator in degree 1.
\end{remark}

We also want to remark that there are some equivalent ways of stating Theorem  \ref{thm:boekstedtDVR} which might be a bit more canonical from a certain point of view.

\begin{prop}
\label{prop:completenessproperties}
In the situation of Theorem \ref{thm:boekstedtDVR} the map $\bS[z] \to A$ extends to a map $\bS_{W(k)}\pow{z} \to A$ by completeness of $A$. The induced canonical maps
\[
\xymatrix{
\THH(A / \bS[z];\Z_p) \ar[r]^\simeq\ar[d]^\simeq &  \THH(A / \bS\pow{z};\Z_p) \ar[d]_\simeq \\  \THH(A / \bS_{W(k)}[z]; \Z_p)\ar[r]^\simeq  &  \THH(A / \bS_{W(k)}\pow{z}; \Z_p) \\
\THH(A / \bS_{W(k)}[z]) \ar[u]_\simeq\ar[r]^\simeq &  \THH(A / \bS_{W(k)}\pow{z}) \ar[u]^\simeq
}
\]
are all equivalences. 
\end{prop}

\begin{proof}
For the upper four maps this follows from the equivalences
\begin{align*}
&\THH(\bS\pow{z} / \bS[z];\Z_p) \simeq \bS\pow{z}^\wedge_p  \\
&\THH(\bS_{W(k)}[z] / \bS[z];\Z_p) \simeq \bS_{W(k)}[z]^\wedge_p \\
&\THH(\bS_{W(k)}\pow{z} / \bS_{W(k)}[z];\Z_p) \simeq \bS_{W(k)}\pow{z} \\
&\THH(\bS_{W(k)}\pow{z} / \bS\pow{z}];\Z_p) \simeq \bS_{W(k)}\pow{z}
\end{align*}
which can all be checked in $\F_p$-homology (see Remark \ref{rem_padic} and the proof of Theorem \ref{thm:boekstedtDVR}).
The last two vertical equivalences follows since $\THH(A / \bS_{W(k)}[z])$ and $\THH(A / \bS_{W(k)}\pow{z})$ are already $p$-complete. If $A$ is of equal characteristic this is clear anyhow (and in the whole diagram we did not need the $p$-completions). In mixed characteristic this follows from Lemma \ref{lem_complete} below, since $A$ is of finite type over $\bS_{W(k)}[z]$ and over $\bS_{W(k)}\pow{z}$, which can be seen by the presentation 
\[
A \cong W(k)[z]/\phi(z) \cong W(k)\pow{z}/ \phi(z) .
\]  
where $\phi$ is the minimal polynomial of the uniformizer $\pi$. 
\end{proof}

Recall that a connective ring spectrum $A$ over a connective, commutative ring spectrum $S$  is said to be of \emph{finite type} if $A$ is as an $R$-module a filtered colimit of perfect modules along increasingly connective maps (i.e. has a cell structure with finite `skeleta'). 
\begin{lem}\label{lem_complete}
If $A$ is $p$-complete and of finite type over $R$ then $\THH(A / R)$ is also $p$-complete $\THH(A/R)$. 
\end{lem}
\begin{proof}
We first observe that all tensor products $A \otimes_R ... \otimes_R A$ are of finite type over $A$ (say by action from the right) which follows inductively.  Thus they are $p$-complete. Finally, the $n$-truncation of $\THH(A / R)$ is equivalent to the $n$-truncation of the restriction of the cyclic Bar construction to $\Delta^{\mathrm{op}}_{\leq n+1}$. This colimit is finite and the stages are $p$-complete by the above.
\end{proof}

We now consider quotients $A'$ of a CDVR $A$ as in Theorem \ref{thm:boekstedtDVR}. Every ideal is of the form $(\pi^k) \subseteq A$ and thus $A' = A/\pi^k$ for some $k \geq 1$.

\begin{prop}\label{quotient}
In the situation above we have a canonical equivalence
\[
\THH(A' / \bS[z]) \simeq \THH(A / \bS[z]; \Z_p) \otimes_{\Z[z]} \HH\big( (\Z[z]/z^k)\,  / \Z[z]\big)
\]
and on homotopy groups we get 
\[
\THH_*(A' /\bS[z]) = A'[x]\langle y \rangle
\]
where $y$ is a divided power generator in degree 2.
\end{prop}
\begin{proof}
Since $\pi$ is a non-zero divisor we can write $A' = A \otimes_{\bS[z]} (\bS[z]/z^k)$ where $\bS[z] / z^k$ is the reduced suspension spectrum of the pointed monoid $\mathbb{N} / [k, \infty)$. Thus we find 
\begin{align*}
\THH(A' / \bS[z])  & \simeq \THH(A / \bS[z]) \otimes_{\bS[z]} \THH( (\bS[z]/z^k) / \bS[z])  \\
& \simeq \THH(A / \bS[z]) \otimes_{\Z[z]} \Big(\Z \otimes_\bS \THH\big( (\bS[z]/z^k) / \bS[z]\big) \Big) \\
& \simeq \THH(A / \bS[z]) \otimes_{\Z[z]}  \HH\big( (\Z[z]/z^k)\,  / \Z[z]\big)\\
& \simeq \THH(A / \bS[z]; \Z_p) \otimes_{\Z[z]}  \HH\big( (\Z[z]/z^k)\,  / \Z[z]\big)
\end{align*}
where in the last step we have used that $p$ is nilpotent in $A'$ and thus we are already $p$-complete. Finally  $\HH\big( (\Z[z]/z^k)\,  / \Z[z]\big)$ is given by a divided power algebra $(\Z[z]/z^k)\langle y \rangle$. To see this we first observe that $\Z[z]/z^k \otimes_{\Z[z]} \Z[z]/z^k$ is given by the exterior algebra
$\Lambda_{\Z[z]/z^k}(e)$ with $e$ in degree $1$. Then it follows that $\HH\big( (\Z[z]/z^k)\,  / \Z[z]\big)$, which is the Bar construction on that,  is given by 
\[
\Tor_*^{\Lambda_{\Z[z]/z^k}(e)} \big(\Z[z]/z^k, \Z[z]/z^k\big)  = (\Z[z]/z^k)\langle y \rangle \ .
\] 
This implies the claim.
\end{proof}

%%%%%%%%%%%%%%%%%%%%%%%%%%%%
\section{Absolute $\THH$ for CDVRs}\label{section5}
%%%%%%%%%%%%%%%%%%%%%%%%%%%%%

For $A$ a CDVR with perfect residue field of characteristic $p$ we have computed $\THH$ relative to $\bS[z]$. In order to compute the absolute $\THH$ we are going to employ a spectral sequence which works very generally (see Proposition \ref{prop:spectralsequencegeneral}).

\begin{prop}
\label{prop:ssconstruction}
For every commutative algebra $A$ (over $\Z$) with an element $\pi \in A$ considered as a $\bS[z]$-algebra
there is a multiplicative, convergent spectral sequence
\[
\THH_*(A / \bS[z]; \Z_p) \otimes_{\Z[z]} \Omega^*_{\Z[z] / \Z} \Rightarrow \THH_*(A; \Z_p).
\]
\end{prop}
\begin{proof}
This is a special case of the spectral sequence of Proposition \ref{prop:spectralsequencegeneral}.
\end{proof}

Now for $A$ a CDVR we want to use this spectral sequence to determine $\THH_*(A; \Z_p)$. From Theorem \ref{thm:boekstedtDVR} we see that this spectral sequence takes the form
\[
E^2 = A[x] \otimes \Lambda(dz) \Rightarrow \THH_*(A; \Z_p)
\]
with $|x| = (2,0)$ and $|dz| = (0,1)$. 
\[
\begin{tikzpicture}[yscale=0.7]
\node(00) at (0,0) {$A$};
\node(20) at (2,0) {$A\{x\}$};
\node(40) at (4,0) {$A\{x^2\}$};
\node at (5,0) {\ldots};
\node(01) at (0,1) {$A\{dz\}$};
\node(21) at (2,1) {$A\{xdz\}$};
\node(41) at (4,1) {\ldots};
\node at (1,0) {$0$};
\node at (1,1) {$0$};
\node at (3,0) {$0$};
\node at (3,1) {$0$};
\node at (0,2) {$0$};
\node at (1,2) {$0$};
\node at (2,2) {$0$};
\node at (3,2) {$\ldots$};
\node at (0,3) {$\vdots$};

\draw[->] (20) edge (01);
\draw[->] (40) edge (21);
\begin{scope}[shift={(-0.6,-0.5)}]
	\draw[-latex] (0,0) -- (7,0);
	\draw[-latex] (0,0) -- (0,4.2);
\end{scope}
\end{tikzpicture}
\]
Using the multiplicative structure one only has to determine a single differential
\[
d^2: A  \{ x\} \to A \{ dz \}  \ .
\]
In the equal characteristic case this has to vanish since $x$ can be chosen to lie in the image of the map $\THH(\F_p) \to \THH(A; \Z_p) \to \THH(A / \bS[z]; \Z_p)$ and thus has to be a permanent cycle. Thus the spectral sequence degenerates and we get $\THH_*(A) = A[x] \otimes \Lambda(dz)$ as there can not be any extension problems for degree reasons. 
\footnote{This can also be seen directly using that $A = k \pow{z} = \F_p \otimes_\bS \bS_{W(k)}\pow{z}$ which implies
\[
\THH(A) = \THH(\F_p) \otimes_{\bS} \THH(\bS_{W(k)}\pow{z}) = \THH(\F_p) \otimes_{\F_p} \HH(A / \F_p) \ .
\]
}

Let us now assume that $A$ is a CDVR of mixed characteristic. Once we have chosen a uniformizer $\pi$ we get a minimal polynomial $\phi(z) \in W(k)[z]$ which we normalize such that 
$\phi(0) = p$. Note that usually $\phi$ is taken to be monic, of the form  $\phi(z) =  z^e + p\theta(z)$. This differs from our convention by the unit $\theta(0)$.
\begin{lem}
\label{lem:boekstedtelem}
There is a choice of generator $x \in \THH_2(A / \bS[z]; \Z_p)$ such that $d^2(x) = \phi'(\pi) dz$. 
\end{lem}
\begin{proof}
$\THH(A;\Z_p)$ agrees with $\THH(A / \bS_{W(k)}; \Z_p)$, since $\THH(\bS_{W(k)};\Z_p) = \bS_{W(k)}$. Since $A$ is of finite type over $\bS_{W(k)}$ we use Lemma \ref{lem_complete} to see that 
$\THH(A / \bS_{W(k)}; \Z_p) = \THH(A/\bS_{W(k)})$. 
For connectivity reasons, 
\[
\THH_1(A/\bS_{W(k)}) = \HH_1(A/W(k)) =  \Omega^1_{A / W(k)}  .
\]
Since $A = W(k)[z]/\phi(z)$, we have 
\[
\Omega^1_{A/W(k)}= {A \{dz \} } / {\phi'(\pi)dz}.
\]
 Comparing with the spectral sequence, this means that the image of $d^2: E^2_{2,0}\to E^2_{0,1}$ is precisely the submodule of $A\{dz\}$ generated by $\phi'(\pi) dz$. Since $A$ is a domain, any two generators of a principal ideal differ by a unit, and thus for any generator $x$ in degree $(2,0)$, $d^2(x)$ differs from $\phi'(\pi) dz$ by a unit. In particular, we can choose $x$ such that $d^2(x) = \phi'(\pi) dz$.
\end{proof}

\begin{remark}\label{rem_generator}
The generator $x\in \THH_2(A / \bS[z]; \Z_p)$ determined by Lemma \ref{lem:boekstedtelem} maps under basechange along $\bS[z]\to \bS$ to a generator of $\THH_2(A/\pi; \Z_p) = \THH_2(k)$. The choice of normalization of $\phi$ with $\phi(0) = p$ is chosen such that this is compatible with the generator obtained  from the generator of $\THH_2(\F_p)$ under the map $\THH_2(\F_p)\to \THH_2(k)$ induced by $\F_p\to k$. 
\end{remark}

Lemma \ref{lem:boekstedtelem} implies that $\THH_*(A, \Z_p)$ is isomorphic to the homology of the DGA 
\[
(A[x]\otimes\Lambda(d\pi), \partial) \qquad\qquad |x| = 2, |d\pi| = 1 
\] with differential 
$\partial x = \phi'(\pi)\cdot d\pi$ and $\partial (d\pi) = 0$ as there are no multiplicative extensions possible. Here we have named the element detected by $dz$ by $d\pi$ as it is given by Connes operator $d: \THH_*(A, \Z_p) \to \THH_*(A, \Z_p)$ applied to the uniformizer $\pi$. This follows from the identification of the degree $1$ part with $\Omega^1_{A/W(k)}$ as in the proof of Lemma \ref{lem:boekstedtelem}. We warn the reader that we have obtained this description for $\THH_*(A; \Z_p)$ from the relative $\THH$ which depends on a choice of uniformizer. As a result the DGA description is only natural in maps that preserve the chosen uniformizer. 

The homology of this DGA can easily be additively evaluated to yield the following result, which was first obtained in \cite[Theorem 5.1]{MR1707702}, but with completely different methods.

\begin{thm}[Lindenstrauss-Madsen]
For a CDVR $A$ of mixed characteristic $(0,p)$ with perfect residue field we have non-natural isomorphisms \footnote{In the sense that they are only natural in maps that preserve the chosen uniformizer.}
\[
\THH_*(A; \Z_p) \cong \begin{cases}
A & \text{for } * = 0 \\
A/ n \phi'(\pi) & \text{for } * = 2n-1 \\
0 & \text{otherwise }
\end{cases}
\]
where $\pi$ is a uniformizer with minimal polynomial $\phi$. 
\end{thm}

In this case the multiplicative structure is necessarily trivial, so that we do not really get more information from the DGA description. But we also obtain a spectral sequence analogous to the one of Proposition \ref{prop:ssconstruction} for $p$-completed $\THH$ of $A$ with coefficients in a discrete $A$-algebra $A'$, which is $\THH(A;\Z_p) \otimes_A A'$. This takes the same form, just base-changed to $A'$. Thus we get the following result, which was of course also known before.

\begin{prop}
\label{prop:thhdga}
For a CDVR $A$ of mixed characteristic and any map of commutative algebras $A \to A'$ we have a non-natural ring isomorphism
\[
\pi_*(\THH(A;\Z_p) \otimes_A  A') \cong H_*(A'[x] \otimes \Lambda(d\pi), \partial)
\]
with $\partial x = \phi'(x)d\pi$ and $\partial d\pi=0$.\qed
\end{prop}

%%%%%%%%%%%%%%%%%%%%%%%%%%%%
\section{Absolute $\THH$ for quotients of DVRs}
\label{sec:thhquot}

Now we come back to the case of quotients of DVRs. Thus let $A' = A/\mathfrak{m}^k \cong A/\pi^k$ where $A$ is a DVR with perfect residue field of characteristic $p$. Recall that in Proposition \ref{quotient} we have shown that 
\[
\THH_*(A' /\bS[z]) \cong A'[x]\langle y \rangle
\]
We want to consider the spectral sequence of Proposition \ref{prop:ssconstruction}, which in this case takes the form
\[
E^2 = A'[x]\langle y \rangle \otimes \Lambda(dz) \Rightarrow \THH_*(A')
\]
with $|x| = (2,0)$, $|y| = (2,0)$ and $|dz| = (0,1)$. 
\[
\begin{tikzpicture}[yscale=0.7, xscale=1.5]
\node(00) at (0,0) {$A'$};
\node(20) at (2,0) {$A'\{x,y\}$};
\node(40) at (4,0) {$A'\{x^2,xy,y^{[2]}\}$};
\node at (5,0) {\ldots};
\node(01) at (0,1) {$A'\{dz\}$};
\node(21) at (2,1) {$A'\{xdz,ydz\}$};
\node(41) at (4,1) {\ldots};
\node at (1,0) {$0$};
\node at (1,1) {$0$};
\node at (3,0) {$0$};
\node at (3,1) {$0$};
\node at (0,2) {$0$};
\node at (1,2) {$0$};
\node at (2,2) {$0$};
\node at (3,2) {$\ldots$};
\node at (0,3) {$\vdots$};
\draw[->] (20) edge (01);
\draw[->] (40) edge (21);
\begin{scope}[shift={(-0.4,-0.5)}]
	\draw[-latex] (0,0) -- (6.3,0);
	\draw[-latex] (0,0) -- (0,4.2);
\end{scope}
\end{tikzpicture}
\]
Here we write $y^{[n]}$ for the $n$-th divided power of $y$. The reader should think of $y^{[n]}$ as `$y^n/n!$'. 

\begin{lem}\label{diff_quot}
We can choose the generator $y$ and its divided powers in such a way that in the associated spectral sequence, $d^2(y^{[i]}) = k \pi^{k-1} \cdot y^{[i-1]} dz$. In particular the differential is a PD derivation, i.e. satisfies $d^2(y^{[i+1]}) = d^2(y) y^{[i]}$ for all $i \geq 0 $. \footnote{Note that since $A'$ is not a domain this does not uniquely determine $y$. One could fix a choice of such a $y$ by comparison with elements in the Bar complex, but this is not necessary for our applications.}
\end{lem}
\begin{proof}
The construction of the spectral sequence of  Proposition \ref{prop:ssconstruction} (given in the proof of  Proposition \ref{prop:spectralsequencegeneral})  applies generally to any $\HH(\Z[z]/\Z)$-module $M$ to produce a spectral sequence
\[
\pi_*(M\otimes_{\HH(\Z[z]/\Z)} \Z[z]) \otimes_{\Z[z]} \HH_*(\Z[z]/\Z) \Rightarrow \pi_*(M). 
\]
Since we can write $A' = A\otimes_{\bS[z]} (\bS[z]/z^k)$, we have
\[
\THH(A') \simeq \THH(A)\otimes_{\THH(\bS[z])} \THH(\bS[z]/z^k) \simeq \THH(A) \otimes_{\HH(\Z[z])} \HH(\Z[z]/z^k).
\]
So we have a map of $\HH(\Z[z])$-algebras $\HH(\Z[z]/z^k) \to \THH(A')$, and thus a multiplicative map of the corresponding spectral sequences. The spectral sequence for $\HH(\Z[z]/z^k)$ is of the form
\[
\HH_*((\Z[z]/z^k) / \Z[z]) \otimes \Lambda(dz) \Rightarrow \HH(\Z[z]/z^k). 
\]
We have that $\HH_*((\Z[z]/z^k) / \Z[z])  = (\Z[z]/z^k)\langle y \rangle$.
Since the spectral sequence is multiplicative, we get
\[
i! d^2(y^{[i]}) = d^2(y^i) = i d^2(y) y^{i-1} = i! d^2(y) y^{[i-1]},
\]
and since the $E^2$-page consists of torsion free abelian groups, we can divide this equation by $i!$ to get
\[
d^2(y^{[i]}) = d^2(y) y^{[i-1]},
\]
i.e. the differential is compatible with the divided power structure.

Now, $\HH_1((\Z[z]/z^k) / \Z[z]) = \Omega^1_{(\Z[z]/z^k)/\Z[z]} = (\Z[z]/z^k)\{dz\} / kz^{k-1}dz$. In particular, in the spectral sequence
\[
\HH_*((\Z[z]/z^k) / \Z[z]) \otimes \Lambda(dz) \Rightarrow \HH(\Z[z]/z^k)
\]
$d^2(y)$ is a unit multiple of $kz^{k-1}dz$. We can thus choose our generator $y$ of $\HH_2((\Z[z]/z^k) / \Z[z])$ in such a way that $d^2(y) = k z^{k-1} dz$, and by compatibility with divided powers, $d^2(y^{[i]}) = k z^{k-1} \cdot y^{[i-1]} dz$. After base-changing along $\Z[z]\to A$, this implies the claim.
\end{proof}

\begin{thm}\label{thm_quot}
Let  $A' \cong A/\pi^k$ be a quotient of a DVR $A$ with perfect residue field of characteristic $p$. Then $\THH_*(A')$ is as a ring non-naturally isomorphic to the homology of the DGA 
\[
(A'[x]\langle y \rangle \otimes\Lambda(d\pi), \partial)  \qquad\qquad |x| = 2, |y| = 2, |d\pi| = 1
\]
with differential $\partial$ given by $\partial(d\pi) = 0$ and $\partial(y^{[i]}) =  k \pi^{k-1} \cdot y^{[i-1]} d\pi$ and 
\[
\partial(x) = \begin{cases}
\phi'(\pi)\cdot d\pi &\text{if $A$ is of mixed characteristic} \\
0 & \text{if $A$ is of equal characteristic}
\end{cases}
\]
Here $\pi \in A$ is a uniformizer and $\phi$ its minimal polynomial.
\end{thm}
\begin{proof}
This follows immediately from Lemma \ref{diff_quot} together with the fact that 
there are no extension problems for degree reasons.
\end{proof}

%%%%%%%%%%%%%%%%%%%%%%%%%%%%%%%%%%%%
\section{Evaluation of the result}\label{eval}
%%%%%%%%%%%%%%%%%%%%%%%%%%%%%%%%%%%%

In this section we want to make the results of Theorem \ref{thm_quot} explicit. 
We start by considering the case of the p-adic integers $\Z_p$ in which Theorem \ref{thm_quot} reduces additively to Brun's result, but  gives some more multiplicative information. We note that all the computations in this section depend on the presentation $A' = A/\pi^k$ and are in particular highly non-natural in $A'$.

\begin{example}\label{ex:integerskbig}
We start by discussing the case $A = \Z_p$ and $k \geq 2$. We pick the  uniformizer $\pi=p$. The minimal polynomial is $\phi(z)=z-p$, and $A' = \Z/p^k$. The resulting groups $\THH_*(\Z/p^k)$ were additively computed by Brun \cite{MR1750729}.

We have $\partial y^{[i]} = kp^{k-1}\partial y^{[i-1]} d\pi$, and since the minimal polynomial is given by $z-p$ we get $\partial x = d\pi$. If $k\geq 2$, then $y' = y-k p^{k-1} x$ still has divided powers, given by
\[
(y')^{[i]} = \sum_{l\geq 0} (-1)^l \frac{k^l p^{l(k-1)}}{l!}y^{[i-l]} x^l,
\]
which makes sense since $v_p(l!) < \frac{l}{p-1}\leq l(k-1)$ by Lemma \ref{lem:legendre} below.

Now $\partial (y')^{[i]} = 0$, and we get a map of DGAs
\[
\left( (\Z/p^k)[x] \otimes\Lambda(d\pi), \partial\right)\otimes_\Z \left(\Z\langle y' \rangle, 0\right) \to \left((\Z/p^k)[x]\langle y \rangle \otimes\Lambda(d\pi), \partial\right) 
\]
which is an isomorphism by a straightforward filtration argument. 
By Proposition \ref{prop:thhdga}, the homology of $((\Z/p^k)[x] \otimes\Lambda(d\pi), \partial)$ coincides with $\pi_*(\THH(\Z_p)\otimes_{\Z_p} \Z/p^k)$. Thus applying the K\"unneth theorem we get 
\[
\THH_*(\Z/p^k) = \pi_*(\THH_*(\Z_p)\otimes_{\Z_p} \Z/p^k)\otimes_\Z \Z\langle y'\rangle
\]
 as rings. Concretely we get 
\begin{align*}
\THH_*(\Z/p^k) &= \bigoplus_{i\geq 0} \pi_{*-2i}(\THH(\Z_p)\otimes_{\Z_p} \Z/p^k)
\\&=\begin{cases} \Z/p^k \oplus \bigoplus_{1\leq i\leq n} \Z/{\gcd(p^k,i)} &\text{for } *=2n\\
\bigoplus_{1\leq i\leq n} \Z/{\gcd(p^k,i)} &\text{for } *=2n-1 \ .
\end{cases}
\end{align*}
\end{example}

So in the case $k\geq 2$, we can replace the divided power generator of our DGA by one in the kernel of $\partial$. We contrast this with the case $k=1$. In this case, of course, we expect to recover B\"okstedt's result $\THH_*(\Z/p) = (\Z/p)[x]$, but it is nevertheless interesting to analyze this result in terms of Theorem \ref{thm_quot} and observe how this differs from Example \ref{ex:integerskbig}.

\begin{example}
\label{ex:integersksmall}
For $A=\Z_p$ with uniformizer $p$ and $k=1$, i.e. $A'=\Z/p$, we have $\partial x = d\pi$ and $\partial y = d\pi$. Here, we can set $x' = x - y$ to obtain an isomorphism of DGAs
\[
\left( \Z[x'] , 0\right)\otimes_\Z \left((\Z/p)\langle y \rangle\otimes\Lambda(d\pi), \partial\right) \to \left((\Z/p)[x]\langle y \rangle \otimes\Lambda(d\pi), \partial\right).
\]
Since $\partial y^{[i]} = y^{[i-1]}dz$, and thus the homology of the second factor is just $\Z/p$ in degree $0$, Künneth applies to show that $\THH_*(\Z/p) \cong (\Z/p)[x']$.
\end{example}

The two qualitatively different behaviours illustrated in Examples \ref{ex:integerskbig} and \ref{ex:integersksmall} also appear in the general case: For sufficiently big $k$, we can modify the divided power generator $y$ to a $y'$ that splits off, and obtain a description in terms of $\THH(A; A')$ (Proposition \ref{prop:kbig}). For sufficiently small $k$, we can modify the polynomial generator to an $x'$ that splits off, and obtain a description in terms of $\HH(A')$ (Proposition \ref{prop:ksmall}. In the general case, as opposed to the case of the integers, these two cases do not cover all possibilities, and for $k$ in a certain region the homology groups of the DGA of Theorem \ref{thm_quot} are possibly without a clean closed form description.

Recall that, in the DGA of Theorem \ref{thm_quot}, we have $\partial x = \phi'(\pi)dz$ and $\partial y = k\pi^{k-1}dz$. The behavior of the DGA depends on which of the two coefficients 
has greater valuation.

\begin{lem}
\label{lem:degreetwo}
In mixed characteristic, we have 
\[
\THH_2(A') \cong A' \oplus A/{\gcd(\phi'(\pi),k\pi^{k-1},\pi^k)}.
\]
\begin{enumerate}
\item\label{item:degen-case} If $p|k$ and $\pi^k | \phi'(\pi)$, we can take as generators
\[
\THH_2(A') \cong A'\left\{x,y\right\}.
\]
\item\label{item:y-case} If $p|k$ and $\phi'(\pi) | \pi^k$, we can take as generators
\[
\THH_2(A') \cong A'\left\{y\right\} \oplus (A/\phi'(\pi))\left\{\frac{\pi^k}{\phi'(\pi)}\right\}.
\]
\item\label{item:yprime-case} If $p\nmid k$ and $\phi'(\pi) | \pi^{k-1}$, we can take as generators
\[
\THH_2(A') \cong A'\left\{y'= y - \frac{k\pi^{k-1}}{\phi'(\pi)} x\right\} \oplus (A/\phi'(\pi))\left\{\frac{\pi^k}{\phi'(\pi)} x\right\}.
\]
\item\label{item:xprime-case} If $p\nmid k$ and $\pi^{k-1} | \phi'(\pi)$, we can take as generators
\[
\THH_2(A') \cong A'\left\{x'= x - \frac{\phi'(\pi)}{k\pi^{k-1}} y\right\} \oplus (A/\pi^{k-1})\left\{\pi y\right\}.
\]
\end{enumerate}
\end{lem}

We now want to discuss the structure of $\THH_*(A')$ in the cases appearing in Lemma \ref{lem:degreetwo}. We start with the simplest case, which is analogous to Example \ref{ex:integersksmall}:

\begin{prop}
\label{prop:ksmall}
Assume we are in the situation of Theorem \ref{thm_quot} and that either $A$ is of equal characteristic, or $A$ is of mixed characteristic and we are in case (\ref{item:degen-case}) or (\ref{item:xprime-case}) of Lemma \ref{lem:degreetwo}, i.e. $p|k$ and $\pi^k | \phi'(\pi)$, or 
$p\nmid k$ and $\pi^{k-1} | \phi'(\pi)$. Then we have
\[
\THH_*(A') \cong \Z[x'] \otimes_\Z H_*\left(A'\langle y \rangle \otimes\Lambda(d\pi), \partial\right)    \qquad |x'| = 2
\]
which evaluates additively to
\begin{gather*}
\THH_{2k}(A'; \Z_p) \cong A/\pi^k \oplus \bigoplus_{i=1}^k A/\gcd(k\pi^{k-1},\pi^k)\\
\THH_{2k-1}(A'; \Z_p) \cong \bigoplus_{i=1}^k A/\gcd(k\pi^{k-1},\pi^k),
\end{gather*}
\end{prop}
\begin{proof}
We set $x'=x$ if $A$ is of equal characteristic or if $p|k$ and $\pi^k | \phi'(\pi)$, and $x' = x - \frac{\phi'(\pi)}{k\pi^{k-1}}y$ if $p\nmid k$ and $\pi^{k-1} | \phi'(\pi)$. Then $\partial x'=0$. We get a map of DGAs
\[
\left( \Z[x'] , 0\right)\otimes_\Z \left(A'\langle y \rangle\otimes\Lambda(d\pi), \partial\right) \to \left(A'[x]\langle y \rangle \otimes\Lambda(d\pi), \partial\right).
\]
which is an isomorphism by a straightforward filtration argument. By Künneth, we get an isomorphism
\[
\THH_*(A') \cong \Z[x'] \otimes_\Z H_*\left(A'\langle y \rangle\otimes\Lambda(d\pi), \partial\right).
\]
The additive description of the homology is easily seen from the fact that $\partial y^{[i]} = k\pi^{k-1} (d\pi) y^{[i-1]}$.
\end{proof}

\begin{remark}
In fact, we can identify $H_*\left(A'\langle y \rangle\otimes\Lambda(d\pi), \partial\right)$ with the Hochschild homology $\HH_*(\Z[z]/z^k\otimes A' / A')$. Compare Section \ref{sec:othersses}. 
\end{remark}

Essentially, the takeaway of Proposition \ref{prop:ksmall} is that in cases (\ref{item:degen-case}) and (\ref{item:xprime-case}) of Lemma \ref{lem:degreetwo} we can modify the polynomial generator $x$ to a cycle which splits a polynomial factor off $\THH(A')$. 

One would hope that, complementarily, in cases \ref{item:y-case} and \ref{item:yprime-case}, we can split off a divided power factor. This is only true after more restrictive conditions. To formulate those, we will require the following lemma on the valuation of factorials:
\begin{lem}[Legendre]
\label{lem:legendre}
For a natural number $l \geq 1$ and a prime $p$ we have
\[
v_p(l!) < \frac{l}{p-1}
\]
\end{lem}
\begin{proof}
We count how often $p$ divides $l!$. Every multiple of $p$ not greater than $l$ provides a factor of $p$, every multiple of $p^2$ provides an additional factor of $p$, and so on. We get the following formula, due to Legendre:
\[
v_p(l!) = \sum_{i\geq 1} \left\lfloor\frac{l}{p^i}\right\rfloor,
\] 
where $\lfloor-\rfloor$ denotes rounding down to the nearest integer. In particular,
\[
v_p(l!) < \sum_{i \geq 1} \frac{l}{p^i} = \frac{l}{p-1}.\qedhere
\]
\end{proof}

\begin{prop}
\label{prop:kbig}
Assume we are in the situation of Theorem \ref{thm_quot}, and for $A$ of equal characteristic $p|k$, and for $A$ of mixed characteristic either $p|k$ (i.e. we are in case (\ref{item:degen-case}) or (\ref{item:y-case}) of Lemma \ref{lem:degreetwo}), or we have the following strengthening of case (\ref{item:yprime-case}): 
\[
v_p\left(\frac{k\pi^{k-1}}{\phi'(\pi)}\right)\geq \frac{1}{p-1}
\]
Then we have an isomorphism of rings
\[
\THH_*(A') \cong \pi_*(\THH(A)\otimes_A A') \otimes_\Z \Z\langle y' \rangle  \qquad |y'| = 2
\]
In particular, we get additively
\begin{gather*}
\THH_{2k}(A'; \Z_p) \cong A/\pi^k \oplus \bigoplus_{i=1}^k A/\gcd(i\phi'(\pi), \pi^k)\\
\THH_{2k-1}(A'; \Z_p) \cong \bigoplus_{i=1}^kA/\gcd(i\phi'(\pi), \pi^k) \ .
\end{gather*}
\end{prop}
\begin{proof}
If $p|k$, all $y^{[i]}$ are cycles, and we set $y' := y$. If
\[
v_p\left(\frac{k\pi^{k-1}}{\phi'(\pi)}\right)\geq \frac{1}{p-1},
\]
we set $y' = y - \frac{k\pi^{k-1}}{\phi'(\pi)}x$. In either case, $(y')$ admits divided powers, defined in the first case just by $(y')^{[i]} = y^{[i]}$, and in the second case by
\[
(y')^{[i]} = \sum_{l\geq 0} (-1)^l\frac{k^l \pi^{l(k-1)}}{\phi'(\pi)^l l!}y^{[i-l]} x^l,
\]
which is well-defined because
\[
v_p\left(\frac{k^l\pi^{l(k-1)}}{\phi'(\pi)^l}\right) \geq \frac{l}{p-1} > v_p(l!)
\]
by assumption and Lemma \ref{lem:legendre}.

We get a map of DGAs
\[
\left( (\Z/p^k)[x] \otimes\Lambda(d\pi), \partial\right)\otimes \left(\Z\langle y' \rangle, 0\right) \to \left((\Z/p^k)[x]\langle y \rangle \otimes\Lambda(d\pi), \partial\right) 
\]
which is an isomorphism by a straightforward filtration argument. By Proposition \ref{prop:thhdga} and K\"unneth, we then get
\[
\THH_*(A'; \Z_p) \cong \pi_*(\THH(A)\otimes_A A')\otimes_\Z \Z\langle y'\rangle\qedhere
\]
\end{proof}

Finally, we want to illustrate that the case `in between' Propositions \ref{prop:kbig} and \ref{prop:ksmall} is more complicated and probably doesn't admit a simple uniform description.

\begin{example}
\label{ex:inbetween}

For a mixed characteristic CDVR $A$ with perfect residue field and $A'=A/\mathfrak{m}^k=A/\pi^k$, Theorem \ref{thm_quot} implies that the even-degree part of $\THH_*(A')$ is given by the kernel of $\partial$ in the DGA $(A'[x]\langle y\rangle \otimes \Lambda(d\pi),\partial)$. We can thus consider $\bigoplus \THH_{2n}(A')$ as a subring of $A'[x]\langle y\rangle$.

Suppose we are in the situation of case (\ref{item:yprime-case}) of Lemma \ref{lem:degreetwo}. Then a basis for $\THH_2(A')$ is given by
\begin{gather*}
y- \frac{k\pi^{k-1}}{\phi'(\pi)} x,\\
\frac{\pi^k}{\phi'(\pi)}x.
\end{gather*}
Now suppose the valuations of the coefficients $\frac{k\pi^{k-1}}{\phi'(\pi)}$ and $\frac{\pi^k}{\phi'(\pi)}$ are positive, but small, say smaller than $\frac{1}{p}$. Then observe that
\[
\left(y- \frac{k\pi^{k-1}}{\phi'(\pi)} x\right)^p = \frac{k^p\pi^{p(k-1)}}{\phi'(\pi)^p} x^p \text{ mod } p,
\]
in particular, under our assumptions, $\left(y- \frac{k\pi^{k-1}}{\phi'(\pi)} x\right)^p$ is divisible by $\pi$ but not $p$. Similarly,
\[
\left(\frac{\pi^k}{\phi'(\pi)}x\right)^p = \frac{\pi^{kp}}{\phi'(\pi)^p} x^p
\]
is divisible by $\pi$ but not $p$. So both of our generators of $\THH_2(A')$ are nilpotent, but cannot admit divided powers. It is not hard to see that this holds more generally for any element of $\THH_2(A')$ that is nonzero mod $\pi$. So in this situation, $\THH_*(A')$ cannot admit a description similar to Proposition \ref{prop:ksmall} or \ref{prop:kbig}.

One example for $A'$ fulfilling the requirements used here is given by $A=\Z_p[\sqrt[e]{p}]$ with uniformizer $\pi=\sqrt[e]{p}$, and $k=e+1$, as long as  $p\nmid  e, k$ and $e>2p$.
\end{example}

%%%%%%%%%%%%%%%%%%%%%
\section{The general spectral sequences}
%%%%%%%%%%%%%%%%%%%%%%

We now want to establish a spectral sequence to compute absolute $\THH$ from relative ones of which Proposition \ref{prop:ssconstruction} is a special case. This will come in two slightly different flavours. We 
let $R \to A$ be a map of commutative rings and let $\bS_R$ be a lift of $R$ to the sphere, i.e. a commutative ring spectrum with an equivalence
\[
\bS_R \otimes_\bS \mathbb{Z} \simeq R. 
\]
The example that will lead to the spectral sequence of Proposition \ref{prop:ssconstruction} is $R = \Z[z]$ and $\bS_R = \bS[z]$. 

Recall that  for every commutative ring $R$ we can form the derived de Rham complex $L\Omega_{R/\Z}$, which has a filtration whose associated graded is in degree $* = i$  given by a shift the non-abelian derived functor of the  $i$-term of the de Rham complex $\Omega^i_{R/\Z}$ (considered as a functor in $R$). Concretely this is done by simplically resolving $R$ by polynomial algebras $\Z[x_1,...,x_k]$, taking $\Omega^i_{\bullet/\Z}$ levelwise and considering the result via Dold-Kan as an object of $\mathcal{D}(\Z)$. This derived functor agrees with the $i$-th derived exterior power $\Lambda^i L_{R/\Z}$ of the cotangent complex $L_{R/\Z}$. For $R$ smooth over $\Z$ this just recovers the usual terms in the de Rham complex. In general one should be aware that $L\Omega_{R/\Z}$ is a filtered chain complex, hence has two degrees, one homological and one filtration degree. We shall only need its associated graded $L\Omega^*_{R/\Z}$ which is a graded chain complex. We warn the reader that the homological direction comes from deriving and has nothing to do with the de Rham differential.

\begin{prop}\label{prop:spectralsequencegeneral}
In the situation described above
there are two multiplicative, convergent spectral sequences
\begin{align*}
\pi_i\Big(\THH(A/\bS_R) \otimes_{R} \HH_j(R / \Z)\Big) \Rightarrow \pi_{i+j} \THH(A)  \\
\pi_i\Big(\THH(A/\bS_R) \otimes_{R} L\Omega^j_{R / \Z}\Big) \Rightarrow \pi_{i+j} \THH(A)  \ .
\end{align*}
Here we use homological Serre grading, i.e. the displayed bigraded ring is the $E_2$-page and the $d^r$-differential has $(i,j)$-bidegree $(-r, r-1)$. A similar spectral sequence with everything all terms $p$-completed (including the tensor products) exists as well.
\end{prop} 
\begin{proof}
We consider the lax symmetric monoidal functor
\begin{align}\label{basechange2}
 \Mod_{\HH(R/\Z)}  &\to \Mod_{\THH(A)}  \\
 M &\mapsto \THH(A) \otimes_{\HH(R/\Z)} M \nonumber
\end{align}
where we have used the equivalence
$
\HH(R/\Z) = \THH(\bS_R) \otimes_\bS \Z
$
to get the $\HH(R/\Z)$-module structure on $\THH(A)$. 

Now we filter $\HH(R/\Z)$ by two different filtrations: either by the 
Whitehead tower
\[
... \to  \tau_{\geq 2} \HH(R/\Z)  \to \tau_{\geq 1} \HH(R/\Z) \to \tau_{\geq 0} \HH(R/\Z) = \HH(R/\Z)
\]
or by the HKR-filtration \cite[Proposition IV.4.1]{NS}
\[
... \to F^2_{\HKR} \to F^1_{\HKR}\to F^0_{\HKR} = \HH(R/\Z) \ .
\]
The HKR-filtration is in fact the derived version of the Whitehead tower, in particular for $R$ smooth (or more generally ind-smooth) both filtrations agree. Both filtrations are complete and multiplicative, in particular they are filtrations through $\HH(R/\Z)$ modules. On the associated graded pieces the $\HH(R/\Z)$-module structure factors through the map $\HH(R/\Z) \to R$ of ring spectra. This is obvious for the Whitead tower and thus also follows for the HKR filtration. Thus the graded pieces are only $R$-modules and as such given by 
$\HH_i(R)$ in the first case and by $\Lambda^j L_{R/ \Z}$ in the second case.

After applying the functor \eqref{basechange2} to this filtration we obtain two multiplicative filtrations of $\THH(A)$:
\[
\THH(A) \otimes_{\HH(R/\Z)}  \big(\tau_{\geq j} \HH(R/\Z)\big)    \qquad \text{and} \qquad  \THH(A) \otimes_{\HH(R/\Z)} F^j_{\HKR}  
\]
which are complete since the connectivity of the pieces tends to infinity. Let us identify the associated gradeds for the HKR filtration, the case of the Whitehead tower works the same:
\begin{align*}
 \THH(A) \otimes_{\HH(R/\Z)} \Lambda^j L_{R/ \Z}& \simeq  \THH(A) \otimes_{\HH(R/\Z)} R \otimes_R \Lambda^j L_{R/ \Z}\\
 & \simeq (\THH(A) \otimes_{\THH(\bS_R)\otimes_{\bS} \Z} (\bS_R \otimes_{\bS} \Z)) \otimes_R \Lambda^j L_{R/ \Z} \\
  & \simeq (\THH(A) \otimes_{\THH(\bS_R)} \bS_R) \otimes_R \Lambda^j L_{R/ \Z} \\
  & \simeq (\THH(A / \bS_R) \otimes_R \Lambda^j L_{R/ \Z} \ .
\end{align*}
Thus by the standard construction we get conditionally convergent, multiplicative spectral sequences which are concentrated in a single quadrant and therefore convergent. 
\end{proof}

If $R$ is smooth (or more generally ind-smooth) over $\Z$ then both spectral sequences of Proposition \ref{prop:spectralsequencegeneral} agree and take the form
\[
\THH_*(A/\bS_R) \otimes_{R} \Omega^*_{R / \Z}\Rightarrow \THH_*(A) \ .
\]
In general the HKR spectral sequence seems to be slightly more useful even though the other one looks easier (at least easier to state).
We will explain the difference in the example of a quotient of a DVR in Section \ref{sec:othersses} where $R = \Z[z]/z^k$ and 
$\bS_R = \bS[z]/z^k$.  

\begin{remark}
With basically the same construction as in Proposition \ref{prop:spectralsequencegeneral} (and if $R\otimes_\Z A$ is discrete in the first case) one gets variants of these spectral sequences which take the form
\begin{align*}
\pi_i\Big(\THH(A/\bS_R) \otimes_{A} \HH_j(R \otimes_\Z A  / A)\Big) \Rightarrow \pi_{i+j} \THH(A)  \\
\pi_i\Big(\THH(A/\bS_R) \otimes_{A} L\Omega^j_{R \otimes_\Z A / A}\Big) \Rightarrow \pi_{i+j} \THH(A) .
\end{align*} 
These spectral sequences agree with the ones of Proposition \ref{prop:spectralsequencegeneral} as soon as $A$ is flat over $R$ or $R$ is smooth over $\Z$, which covers all cases of interest for us. These modified spectral sequences are probably in general the `correct' ones but we have decided to state Proposition \ref{prop:spectralsequencegeneral} in the more basic form.
\end{remark}

Finally we end this section by construction a slightly different spectral sequence in the situation of a map of rings $A \to A'$. This was constructed in Theorem 3.1 of \cite{MR1750728}.
See also Brun \cite{MR1750729}, which contains the special case $A=\Z_p$. We will explain how it was used by Brun to compute $\THH_*(\Z/ p^n)$ in the next section and compare that approach to ours. 
\begin{prop}\label{ss_basis}
In general for a map of rings $A \to A'$ there is a multiplicative, convergent spectral sequence
\[
\pi_i\Big( \HH(A' /A) \otimes_{A'}  \pi_j\big(\THH(A) \otimes_A A'\big)\Big) \Rightarrow \THH_{i+j}(A').
\] 
\end{prop}
\begin{proof}
We filter $\THH(A) \otimes_A A' =: T$ by its Whitehead tower $\tau_{\geq \bullet} T $ and consider the associated filtration
\[
\THH(A') \otimes_{T} \tau_{\geq \bullet}T  .
\]
This filtration is multiplicative, complete and the colimit is given by $\THH(A')$. The associated graded is given by 
\begin{align*}
\THH(A') \otimes_{T} \pi_jT &  \simeq \THH(A') \otimes_{T} A' \otimes_{A'} \pi_j T  \\
&  \simeq \left( \THH(A') \otimes_{\THH(A) \otimes_A A'} A' \right)\otimes_{A'} \pi_j T  \\
&  \simeq
\left(\THH(A') \otimes_{ \THH(A)} A\right) \otimes_{A'} \pi_j T  \\
& \simeq \HH(A' / A) \otimes_{A'} \pi_j T \ .
\end{align*}
where we have again used various base change formulas for $\THH$. 
\end{proof}

\section{Comparison of spectral sequences}
\label{sec:othersses}

Let us consider the situation of Section  \ref{sec:thhquot} i.e. $A' = A/\pi^k$ is a quotient of a DVR $A$ with perfect residue field of characteristic $p$. We want to compare four different multiplicative spectral sequences converging to $\THH(A')$ that can be used in such a situation. 
They all have absolutely isomorphic (virtual) $E^0$-pages give by $A[x]\langle y \rangle \otimes \Lambda(dz)$ but totally different grading and differential structure.
\begin{enumerate}
\item
In Section 5  we have constructed a spectral sequence which ultimately identifies $\THH_*(A')$ as the homology of a DGA $(A'[x]\langle y\rangle \otimes \Lambda(d\pi),\partial)$, see Theorem \ref{thm_quot}.  This spectral sequence takes the form
\[
\begin{tikzpicture}[yscale=0.7, xscale=1.3]
\node(00) at (0,0) {$1$};
\node(20) at (2,0) {$x,y$};
\node(21) at (2,1) {$xdz, ydz$};
\node(40) at (4,0) {$x^2, xy, y^{[2]}$};
\node at (3,0) {$0$};
\node at (0,2) {$0$};
\node at (1,2) {$0$};
\node at (1,0) {$0$};
\node at (1,1) {$0$};
\node at (2,2) {$\ldots$};
\node at (5,0) {$\ldots$};
\node at (3,1) {$\ldots$};
\node(01) at (0,1) {$dz$};
\draw[->] (20) edge (01);
\draw[->] (40) edge (21);
\begin{scope}[shift={(-0.4,-0.5)}]
	\draw[-latex] (0,0) -- (6.2,0);
	\draw[-latex] (0,0) -- (0,3.2);
\end{scope}
\end{tikzpicture}
\]
i.e. we have both $x$ and $y$ along the lower edge, and they both support differentials hitting certain multiples of $dz$ (here $dz$ corresponds to $d\pi$). The main point is that it suffices to determine the differential on $x$ and $y$ and the rest follows using multiplicative and divided power structures. There is no space for higher differentials.
\item
We now consider Brun's spectral sequence, see Proposition \ref{ss_basis}. It also computes $\THH_*(A')$ but has $E^2$-term
\[
E^2 = \HH_*(A'/A) \otimes_{A'} \pi_*( \THH_*(A) \otimes_A A') 
\]
Since $\HH_*(A'/A)$ is a divided power algebra $A'\langle y\rangle$, and $\pi_*( \THH_*(A) \otimes_A A')$  can be computed as the homology of the DGA $(A'[x]\otimes \Lambda(dz),\partial)$ by Proposition \ref{prop:thhdga}, one can introduce a virtual zeroth page
of the form 
\[
E^0 = A'[x]\langle y \rangle \otimes \Lambda(dz) \qquad |y| = (2,0), |x| = (0,2), |dz| = (0,1)
\]
We interpret $\partial$ as the $d^0$-differential\footnote{We do not claim that there is a direct algebraic construction of a spectral sequence with this zeroth page. We simply define the spectral sequence by defining $E^0$ and $d^0$ as explained and from $E^2$ and higher on we take Brun's spectral sequence. This should be seen as a mere tool of visualization.} and get the following picture:

\[
\begin{tikzpicture}[yscale=0.7, xscale=1]
\node(00) at (0,0) {$1$};
\node(20) at (2,0) {$y$};
\node(02) at (0,2) {$x$};
\node at (0,5) {$\vdots$};
\node(03) at (0,3) {$xdz$};
\node(04) at (0,4) {$x^2$};
\node(01) at (0,1) {$dz$};
\node at (3,0) {$0$};
\node(40) at (4,0) {$y^{[2]}$};
\node(21) at (2,1) {$ydz$};
\node at (3,1) {$\ldots$};
\node at (2,2) {$\vdots$};
\node at (5,0) {$\ldots$};
\node at (1,3) {$0$};
\node at (1,4) {$\vdots$};

\draw[->] (20) edge (01);
\draw[->] (40) edge (21);
\draw[->] (02) edge (01);
\draw[->] (04) edge (03);
\node at (1,0) {$0$};
\node at (1,1) {$0$};
\node at (1,2) {$0$};

\begin{scope}[shift={(-0.4,-0.5)}]
	\draw[-latex] (0,0) -- (6.3,0);
	\draw[-latex] (0,0) -- (0,6.2);
\end{scope}
\end{tikzpicture}
\]
This spectral sequence behaves well and degenerates in the `big $k$' case discussed in \ref{prop:kbig}, since then we have divided power elements $(y')^{[i]}\in \THH_*(A')$ that are detected by the $y^{[i]}$, but this is not obvious from this spectral sequence, and Brun \cite{MR1750729} has to do serious work to determine its structure in the case $A'=\Z/p^k$ for $k\geq 2$.

In fact, for $A'=\Z/p^k$ with $k=1$ the spectral sequence becomes highly nontrivial. After $d^0$, determined by $d^0(x) = dz$, the leftmost column consists of elements of the form $x^{ip}$ and $x^{ip-1} dz$. From Example \ref{ex:integersksmall}, we know that $\THH_*(\F_p)$ is polynomial on $x' = x-y$. This is detected as $y$ in this spectral sequence. Since $y$ is a divided power generator, its $p$-th power is $0$ on the $E_\infty$-page. But $p^{k-1}(x-y)^p = p^{k-1}x^p$, and thus there is a multiplicative extension. In addition, the elements $x^{kp-1} dz$ and the divided powers of $y$ cannot exist on the $E_\infty$-page, so there are also longer differentials.

While these phenomena might seem like a pathology in the case $A=\Z_p$ -- after all, we knew $\THH(\Z/p)$ before -- qualitatively, they generally appear whenever we are not in the `big $k$' case discussed in Proposition \ref{prop:ksmall}.
\item
We can also consider the first spectral constructed in Proposition \ref{prop:spectralsequencegeneral}, which takes the form
\[
E^2 = \THH_*(A' /(\bS[z]/z^k)) \otimes_{A'} \HH_i((A'[z]/z^k)/A') \Rightarrow \THH_*(A'). 
\]
One gets $\THH_*(A' / (\bS[z]/z^k) ) = A'[x]$ by a version of Theorem \ref{thm:boekstedtDVR}, and $\HH((A'[z]/z^k)/A')$ is computed as the homology of the DGA $(A'\langle y\rangle \otimes \Lambda(dz),\partial)$ where $y$ sits in degree $2$ and $dz$ in degree 1. Thus we again introduce a virtual $E^0$-term 
\[
E^0  = A'[x]\langle y \rangle \otimes \Lambda(dz) \qquad |y| = (0,2), |x| = (2,0), |dz| = (0,1)
\]
and consider $\partial$ as a $d^0$ differential. Then 
 the spectral sequence visually looks as follows:

\[
\begin{tikzpicture}[yscale=0.7, xscale=1]
\node(00) at (0,0) {$1$};
\node(20) at (2,0) {$x$};
\node(02) at (0,2) {$y$};
\node at (0,3) {$\vdots$};
\node(01) at (0,1) {$dz$};
\node at (3,0) {$0$};
\node(40) at (4,0) {$x^{2}$};
\node(21) at (2,1) {$xdz$};
\node at (3,1) {$\ldots$};
\node at (2,2) {$\ldots$};
\node at (5,0) {$\ldots$};

\draw[->] (20) edge (01);
\draw[->] (02) edge (01);
\node at (1,0) {$0$};
\node at (1,1) {$0$};
\node at (1,2) {$0$};

\begin{scope}[shift={(-0.4,-0.5)}]
	\draw[-latex] (0,0) -- (6.3,0);
	\draw[-latex] (0,0) -- (0,4.2);
\end{scope}
\end{tikzpicture}
\]
This spectral sequence behaves well and degenerates in the `small $k$' case discussed in Proposition \ref{prop:ksmall}, since then we have a polynomial generator $x'\in \THH_*(A')$ whose powers are detected by the $x^i$. If we are not in this case, we generally have nontrivial extensions. For example, let $A'$ be chosen such that $p\nmid k$, and $\pi\phi'(\pi) | \pi^{k-1}$. In this case, $\THH_2(A')$, using \ref{thm_quot}, is of the form
\[
A'\{y'\} \oplus (A/\phi'(\pi))\left\{\frac{\pi^k}{\phi'(\pi)} x\right\},
\] 
with 
\[
y' = y - \frac{k\pi^{k-1}}{\phi'(\pi)} x.
\]
In this spectral sequence, the $E^\infty$ page consists in total degree $2$ of a copy of $(A/\pi^{k-1})\{\pi y\}$ in degree $(0,2)$, and a copy of $(A/\pi\phi'(\pi))\left\{\frac{\pi^{k-1}}{\phi'(\pi)}x\right\}$ in degree $(2,0)$. The element $y'\in \THH_2(A')$ is detected as a generator of the degree $(2,0)$ part, but it is not actually annihilated by $\pi\phi'(\pi)$. Rather, $\pi\phi'(\pi)y'$ agrees with $\pi\phi'(\pi)y$, detected as a $\phi'(\pi)$-multiple of the generator in degree $(0,2)$ and nonzero under our assumption $\pi\phi'(\pi) | k\pi^{k-1}$.

\item
Finally we can consider the second spectral sequence constructed in Proposition \ref{prop:spectralsequencegeneral} which takes the form
\[
E^2 = \THH(A' / ( \bS[z]/z^k) ) \otimes_{A'} L\Omega_{A'/A} \Rightarrow \THH(A'). 
\]
One again has $\THH_*(A' / (\bS[z]/z^k) ) = A'[x]$ and $L\Omega_{A'/A}$ is computed as the homology of the DGA $(A'\langle y\rangle \otimes \Lambda(dz),\partial)$ where this time $y$ sits in grading $1$ and homological degree $1$ (recall that $L\Omega$ has a grading and a homological degree). Thus our virtual $E^0$-term this time takes the form
\[
E^0  = A'[x]\langle y \rangle \otimes \Lambda(dz) \qquad |y| = (1,1), |x| = (2,0), |dz| = (0,1) \ .
\]
and the differential $\partial$ becomes a $d^1$.
The spectral sequence looks graphically as follows:

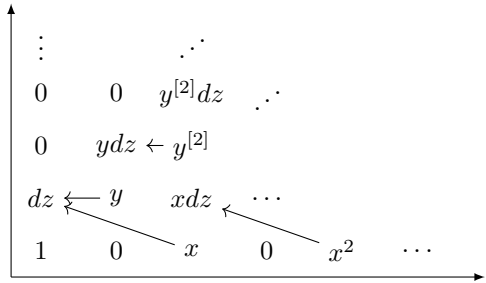
\[
\begin{tikzpicture}[yscale=0.7, xscale=1]
\node(00) at (0,0) {$1$};
\node(20) at (2,0) {$x$};
\node(11) at (1,1) {$y$};
\node at (0,3) {$0$};
\node at (1,3) {$0$};
\node at (2,4) {$\iddots$};
\node at (3,3) {$\iddots$};
\node at (0,4) {$\vdots$};
\node(01) at (0,1) {$dz$};
\node at (3,0) {$0$};
\node(40) at (4,0) {$x^{2}$};
\node(21) at (2,1) {$xdz$};
\node at (3,1) {$\ldots$};
\node(22) at (2,2) {$y^{[2]}$};
\node at (5,0) {$\ldots$};
\node(12) at (1,2) {$ydz$};
\node at (2,3) {$y^{[2]}dz$};

\draw[->] (11) edge (01);
\draw[->] (22) edge (12);
\draw[->] (20) edge (01);
\draw[->] (40) edge (21);
\node at (1,0) {$0$};
\node at (0,2) {$0$};

\begin{scope}[shift={(-0.4,-0.5)}]
	\draw[-latex] (0,0) -- (6.3,0);
	\draw[-latex] (0,0) -- (0,5.2);
\end{scope}
\end{tikzpicture}
\]
This spectral sequence is a slightly improved version of spectral sequence (3) as there are way less higher differentials possible. The whole wedge above the diagonal line through $1$ on the $j$-axis is zero. Again this spectral sequence behaves well and degenerates in the `small $k$' case \ref{prop:ksmall}, but behaves as badly in the other cases.

\end{enumerate}
Essentially, one should view Proposition \ref{prop:ksmall} as degeneration result for the spectral sequences (3) and (4), and Proposition \ref{prop:kbig} as a degeneration result for the Brun spectral sequence (2). By putting both the Bökstedt element $x$ and the divided power element $y$ (coming from the relation $\pi^k=0$) in the same filtration, the spectral sequence (1) that we have used  allows us to uniformly treat both of these cases, as well as still behaving well in the cases not covered by Propositions \ref{prop:kbig} and \ref{prop:ksmall} (like Example \ref{ex:inbetween}), where the homology of the DGA of Theorem \ref{thm_quot} becomes more complicated and all of the three alternative spectral sequences discussed here can have nontrivial extension problems, seen in our spectral sequence in the form of cycles which are interesting linear combinations of powers of $x$ and $y$.

%%%%%%%%%%%%%%%%%%%%%
\section{B\"okstedt periodicity for complete regular local rings}\label{sec_complete}
%%%%%%%%%%%%%%%%%%%%%%

In this section we want to discuss the more general case of a complete regular  local ring $A$, that is, a complete local ring $A$ whose maximal ideal $\mathfrak{m}$ is generated by a regular sequence $(a_1,\ldots, a_n)$, see \cite[\href{https://stacks.math.columbia.edu/tag/00NQ}{Tag 00NQ}]{stacks-project} and \cite[\href{https://stacks.math.columbia.edu/tag/00NU}{Tag 00NU}]{stacks-project}.
Assume furthermore that $A/\mathfrak{m}=k$ is perfect of characteristic $p$. We focus on the mixed characteristic case, since by a result of Cohen \cite{MR16094}, $A$ agrees with a power series ring over $k$ in the equal characteristic case.

We can regard $A$ as an algebra over $\bS[z_1,\ldots,z_n] = \bS[\mathbb{N}\times \ldots\times \mathbb{N}]$. We then have the following generalisation of Theorem \ref{thm:boekstedtDVR}:
\begin{thm} \label{thm_regularlocal} 
For a complete regular  local ring $A$ of mixed characteristic with perfect residue field of characteristic $p$ we have
\[
\THH_*(A/\bS[z_1,\ldots,z_n];\Z_p) = A[x]
\]
with $x$ in degree $2$.
\end{thm}

We will give a  proof which is completely analogous to the one of Theorem \ref{thm:boekstedtDVR}. We first need the following Lemma. 

\begin{lem}
\label{lem:completeregularpresentation}
If $A$ is a complete regular local ring as above, it is of finite type over $W(k)\pow{z_1,\ldots,z_n}$. More precisely, it takes the form
\[
A=W(k)\pow{z_1,\ldots,z_n}/\phi(z_1,\ldots,z_n)
\]
for a power series $\phi$ with $\phi(0,\ldots,0)=p$.
\end{lem}
\begin{proof}
The map $W(k)\pow{z_1,\ldots,z_n}\to A$ is a surjective $W(k)\pow{z_1,\ldots,z_n}$-module map, and its kernel $K$ base-changes along $W(k)\pow{z_1,\ldots,z_n}\to W(k)$ to the kernel of $W(k)\to k$, i.e. $pW(k)$. $K$ is therefore free of rank $1$, on a generator $\phi\in W(k)\pow{z_1,\ldots,z_n}$ reducing to $p$ modulo $(z_1,\ldots,z_n)$.
\end{proof}

\begin{proof}[Proof of Theorem \ref{thm_regularlocal}]
From Lemma \ref{lem:completeregularpresentation}, one can deduce as in Proposition \ref{prop:completenessproperties} that the following all agree:
\[
\begin{tikzcd}
\THH(A/\bS[z_1,\ldots,z_n];\Z_p)\rar{\simeq}\dar{\simeq} & \THH(A/\bS\pow{z_1,\ldots,z_n};\Z_p)\dar{\simeq}\\
\THH(A/\bS_{W(k)}[z_1,\ldots,z_n];\Z_p)\rar{\simeq} & \THH(A/\bS_{W(k)}\pow{z_1,\ldots,z_n};\Z_p)\\
    & \THH(A/\bS_{W(k)}\pow{z_1,\ldots,z_n})\uar{\simeq}
\end{tikzcd}
\]
These statements can again all be checked modulo $p$, observing that the lower right hand term $\THH(A/\bS_{W(k)}\pow{z_1,\ldots,z_n})$ is already $p$-complete by Lemma \ref{lem_complete}  since $A$ is of finite type over $\bS_{W(k)}\pow{z_1,\ldots,z_n}$. The key is (as in the proof of Proposition \ref{prop:completenessproperties}) that the maps
\[
\xymatrix{
\F_p[z_1,\ldots, z_n] \ar[r]\ar[d] &  \F_p\pow{z_1,\ldots, z_n}\ar[d] \\
k[z_1,\ldots, z_n] \ar[r] &  k\pow{z_1,\ldots, z_n}
}
\]
are all relatively perfect.

Note that, as opposed to the DVR case, $A$ is not of finite type over the ring spectrum $\bS_{W(k)}[z_1,\ldots,z_n]$, and thus $\THH(A/\bS_{W(k)}[z_1,\ldots,z_n])$ is not necessarily $p$-complete.
\end{proof}

From Proposition \ref{prop:spectralsequencegeneral}, we now obtain:
\begin{prop}
There is a multiplicative, convergent spectral sequence
\[
\THH_*(A/\bS[z_1,\ldots,z_n];\Z_p)\otimes_\Z \Omega^*_{\Z[z_1,\ldots,z_n]/\Z} \Rightarrow \THH_*(A;\Z_p).
\]
\end{prop}

Analogously to Lemma \ref{lem:boekstedtelem} we can describe the differential $d^2$:

\begin{lem}
We can choose the generator $x\in \THH_2(A/\bS[z_1,\ldots,z_n];\Z_p)$ in such a way that
\[
 d^2 x =\sum_i \frac{\partial \phi}{\partial z_i} dz_i.
\]
\end{lem}
\begin{proof}
We have $\THH_1(A;\Z_p) = \Omega^1_{A/W(k)\pow{z_1,\ldots,z_n}}$. For a polynomial $\phi$ as in Lemma \ref{lem:completeregularpresentation}, we get
\[
\Omega^1_{A/W(k)\pow{z_1,\ldots,z_n}} = A\{dz_1,\ldots,dz_n\} \Big/ \left(\sum_i \frac{\partial \phi}{\partial z_i} dz_i\right).
\]
So the image of $d^2$ in degree $(0,1)$ has to agree with the ideal generated by $\sum_i \frac{\partial \phi}{\partial z_i} dz_i$. Up to a unit, we thus have
\[
d^2 x = \sum_i \frac{\partial \phi}{\partial z_i} dz_i.\qedhere
\]
\end{proof}

For $n=2$, this differential again completely determines $\THH_*(A;\Z_p)$, since $d^2$ in degrees $(2k,0)\mapsto (2k-2,1)$ is injective, and therefore the $E^3$-page is concentrated in degrees $(0,0)$, $(2k,1)$ and $(2k,2)$ for $k\geq 0$ and the spectral sequence degenerates thereafter without potential for extensions. For $n\geq 3$, there could be extensions, and for $n\geq 4$, there could be longer differentials, both of which we do not know how to control.

Finally, we want to remark a couple of things about computing $\THH_*(A';\Z_p)$ for $A' = A/(f_1,\ldots,f_d)$, with $(f_1,\ldots,f_d)$ a regular sequence analogously to Section \ref{sec:thhquot}. We still have a spectral sequence
\[
\THH_*(A'/\bS[z_1,\ldots,z_n];\Z_p) \otimes_\Z \HH_*(\Z[z_1,\ldots,z_n]) \Rightarrow \THH_*(A'),
\]
but the study of $\THH_*(A'/\bS[z_1,\ldots,z_n];\Z_p)$ turns out to be potentially more subtle. As opposed to Proposition \ref{quotient}, we only have a spectral sequence
\[
\THH_*(A/\bS[z_1,\ldots,z_n];\Z_p) \otimes_A \HH_*(A' / A) \Rightarrow \THH_*(A'/\bS[z_1,\ldots,z_n];\Z_p),
\]
but this does not necessarily degenerate into an equivalence since there is no analogue of the spherical lift $\bS[z]/z^k$ used in the proof of Proposition \ref{quotient}.

In our case, $\THH_*(A/\bS[z_1,\ldots,z_n];\Z_p)$ is $A[x]$, and $\HH_*(A' / A)$ is easily seen to be a divided power algebra on $d$ generators. So the spectral sequence is even and cannot have nontrivial differentials. However, there is potential for multiplicative extensions. We have been informed by Guozhen Wang that these do indeed show up, which will be part of forthcoming work of Guozhen Wang with Ruochuan Liu.

\section{Logarithmic THH of CDVRs} \label{logTHH}

In this section we want to explain how to deduce results about logarithmic topological Hochschild homology from our methods. This way we recover known computations of Hesselholt--Madsen \cite{MR1998478} for logarithmic $\THH$ of DVRs.  We thank Eva H\"oning for asking about the relation between relative and logarithmic $\THH$, which inspired this section. 

First we recall the definition of logarithmic $\THH$ following \cite{MR1998478, MR3941522} and \cite{MR2544395}. For an abelian monoid  $M$  we consider the spherical group ring 
$\bS[M]$ and have 
\[
\THH(\bS[M]) = \bS[ B^\mathrm{cyc} M]
\]
where $B^{\mathrm{cyc}}M$ is the cyclic Bar construction, i.e. the unstable version of topological Hochschild homology. 
We denote by $M \to M^{\gp}$ the group completion and define the logarithmic $\THH$ of $\bS[M]$ relative to $M$ by
\[
\THH\big(\bS[M]\, |\, M\big) := \bS[M \times_{M^{\gp}}  B^\mathrm{cyc} M^{\gp} ] \ .
\]
There are induced maps of commutative  ring spectra 
\[
\THH(\bS[M]) \to \THH\big(\bS[M] \,|\, M\big) \to \bS[M] 
\]
whose composition is the canonical map. These are induced from the maps $B^\mathrm{cyc} M\to M \times_{M^{\gp}}  B^\mathrm{cyc} M^{\gp} \to M$.
\begin{defn}
For a commutative ring $R$ with a map $\bS[M] \to R$ 
we define \emph{logarithmic THH} as the commutative ring spectrum
\[
\THH(R\, |\, M) := \THH(R) \otimes_{\THH(\bS[M])} \THH\big(\bS[M]\, |\,  M\big) \ .
\]
\end{defn}
In practice, we will only need the case $M = \mathbb{N}$ with the map 
$\bS[\mathbb{N}] = \bS[z] \to R$ given by sending $z$ to an element $\pi \in R$. In this case we will also denote $\THH(R\,|\,\mathbb{N})$ by $\THH(R\,|\,\pi)$.

\begin{lem}\label{lem_log}
We have an equivalence of commutative ring spectra
\[
\THH(R / \bS[M]) \simeq \THH(R\, |\, M)  \otimes_{\THH(\bS[M]\, |\, M)} \bS[M] \ .
\]
\end{lem}
\begin{proof}
We have 
\begin{align*}
\THH(R / \bS[M]) & \simeq \THH(R) \otimes_{\THH(\bS[M])}  \bS[M] \\
& \simeq \THH(R) \otimes_{\THH(\bS[M])} \THH\big(\bS[M]\, |\,  M\big) \otimes_{\THH(\bS[M]\, |\, M)} \bS[M]  \\
& \simeq \THH(R\, |\, M)  \otimes_{\THH(\bS[M]\, |\, M)} \bS[M] \ .\qedhere
\end{align*}
\end{proof}

We will use this Lemma to get a spectral sequence similar to the one of Proposition \ref{prop:spectralsequencegeneral}. To this end let us introduce some further notation. We set
\[
\HH( \Z[M]\, |\, M) := \THH\big(\bS[M]\, |\, M\big) \otimes_{\bS} \Z = \Z[M \times_{M^{\gp}}  B^\mathrm{cyc} M^{\gp} ],
\]
which comes with a canonical map $\HH(\Z[M]) \to \HH( \Z[M]\, |\, M)$\ .
\begin{example}
For $M = \mathbb{N}$ we have $\Z[M] = \Z[z]$ and we get that the logarithmic Hochschild homology
$\HH_*(\Z[M]\, |\, M) = \HH_*(\Z[z]\,|\, z)$ is the exterior algebra over $\Z[z]$ on a generator $\dlog z$.
One should think of $\dlog z$ as `$dz / z$'. Indeed, under the canonical map
\[
\Omega^*_{\Z[z] / \Z} = \HH_*(\Z[z]) \to \HH_*(\Z[z]\, |\, z)
\]
the element $dz \in \Omega^1_{\Z[z]/\Z}$ gets mapped to $z\cdot \dlog z$ as one easily checks. In particular one should think of $\HH_*(\Z[z]\, |\, z)$ as differential forms on the space $\mathbb{A}^1 \setminus 0$ with logarithmic poles at $0$. This is a subalgebra of differential forms on $\mathbb{A}^1 \setminus 0$ as is topologically witnessed by the injective map 
$\HH_*(\Z[z]\, |\, z) \to \HH_*(\Z[z^\pm])$ and the map $\HH_*(\Z[z])$ then includes the forms on $\mathbb{A}^1$. 
\end{example}

\begin{prop}\label{SS_log}
For every map $\bS[M] \to R$ of commutative rings there is a multiplicative and convergent spectral sequence 
\begin{equation*}
\pi_i\Big(\THH(R/\bS[M]) \otimes_{\Z[M]} \HH_j(\Z[M]\, |\, M)\Big) \Rightarrow \pi_{i+j} \THH(R\, |\, M) \ .
\end{equation*}
Moreover this spectral sequence receives a multiplicative map from the spectral sequence 
\[
\pi_i\Big(\THH(R/\bS[M]) \otimes_{\Z[M]} \HH_j(\Z[M])\Big) \Rightarrow \pi_{i+j} \THH(R) 
\]
of Proposition \ref{prop:spectralsequencegeneral}, which refines on the abutment the canonical map $\THH_*(R) \to \THH_*(R\, |\, M)$ and on the $E^2$-page the map $\HH_*(\Z[M]) \to \HH_*(\Z[M]\,|\, M)$. Similarly, there is a $p$-completed version of this spectral sequence. 
\end{prop}
\begin{proof}
We proceed exactly as in the proof of Proposition \ref{prop:spectralsequencegeneral} and define a filtration on $\THH(R\, |\, M)$ by
\[
\THH(R\, |\, M) \otimes_{\HH(\Z[M]\, |\, M)}  \tau_{\geq i}{\HH(\Z[M]\, |\, M)}  . 
\]
By the same manipulations as there we get the result using Lemma \ref{lem_log}. 
\end{proof}

Now for a CDVR $A$ of mixed characteristic with perfect residue field of characteristic $p$, we want to use this spectral sequence to determine the logarithmic $\THH_*(A\,|\,\pi ; \Z_p)$. As usual, this denotes the homotopy groups of the $p$-completion of $\THH(A \,|\, \pi)$.

 From Theorem \ref{thm:boekstedtDVR} we see that the spectral sequence of Proposition \ref{SS_log} takes the form
\[
E^2 = A[x] \otimes \Lambda(\dlog z) \Rightarrow \THH_*(A\,|\,\pi ; \Z_p)
\]
with $|x| = (2,0)$ and $|\dlog z| = (0,1)$. 
\[
\begin{tikzpicture}[yscale=0.7, xscale=1.4]
\node(00) at (0,0) {$A$};
\node(20) at (2,0) {$A\{x\}$};
\node(40) at (4,0) {$A\{x^2\}$};
\node at (5,0) {\ldots};
\node(01) at (0,1) {$A\{\dlog z\}$};
\node(21) at (2,1) {$A\{x\dlog z\}$};
\node(41) at (4,1) {\ldots};
\node at (1,0) {$0$};
\node at (1,1) {$0$};
\node at (3,0) {$0$};
\node at (3,1) {$0$};
\node at (0,2) {$0$};
\node at (1,2) {$0$};
\node at (2,2) {$0$};
\node at (3,2) {$\ldots$};
\node at (0,3) {$\vdots$};

\draw[->] (20) edge (01);
\draw[->] (40) edge (21);
\begin{scope}[shift={(-0.6,-0.5)}]
	\draw[-latex] (0,0) -- (7,0);
	\draw[-latex] (0,0) -- (0,4.2);
\end{scope}
\end{tikzpicture}
\]
The spectral sequence receives a map from the spectral sequence
\[
E^2 = A[x] \otimes \Lambda(dz) \Rightarrow \THH_*(A; \Z_p)
\]
used in Section \ref{section5}. This map sends $x$ to $x$ and $dz$ to $\pi \dlog z$. Thus from our knowledge of the differential in this second spectral sequence where we have $d^2(x) = \phi'(\pi) dz$ (Lemma \ref{lem:boekstedtelem}), we can conclude that $d^2$ in the first spectral sequence has to send $x$ to $\pi \phi'(\pi) \dlog z$. Thus 
we get the following result of Hesselholt--Madsen \cite[Theorem 2.4.1 and Remark 2.4.2]{MR1998478}.
\begin{prop}
For a CDVR $A$ of mixed characteristic with perfect residue field of characteristic $p$, the ring $\THH_*(A\,|\, \pi ; \Z_p)$ is isomorphic to the homology of the DGA
\[
H_*(A[x] \otimes \Lambda(\dlog\pi), \partial)
\]
with $\partial x = \pi \phi'(x)\dlog\pi$ and $\partial d\pi=0$. In particular
\[
\THH_*(A\,|\, \pi; \Z_p) = \begin{cases}
A & \text{for } * = 0 \\
A/  n \pi \phi'(\pi) & \text{for } * = 2n-1 \\
0 & \text{otherwise }
\end{cases}
\]
Similarly to Proposition \ref{prop:thhdga} one can also obtain a version with coefficients in an $A$-algebra $A'$, namely that $\pi_*( \THH(A \,|\, \pi; \Z_p) \otimes_A A')$ is given by the homology of the DGA 
$H_*(A'[x] \otimes \Lambda(\dlog\pi), \partial)$ with $\partial$ as above.
\end{prop}

Note that one could alternatively also deduce the differential in the log spectral sequence using the description of $\THH_1(A \,|\, \pi; \Z_p)$ in terms of logarithmic K\"ahler differentials, similar to the way we have deduced the differential in the absolut spectral sequence for $\THH_*(A;\Z_p)$ in Lemma \ref{lem:boekstedtelem}.

\begin{remark}
We have considered the DVR $A$ together with the map $\mathbb{N} \to A$ as input for our logarithmic $\THH$. This is what is called a pre-log ring. The associated log ring is given by the saturation $M \to A$ with $M = A \cap (A[\pi^{-1}])^\times$. However we have $M = A^\times \times \mathbb{N}$ as one easily verifies. Chasing through the definitions one sees that this implies that $\THH(A \,|\, \mathbb{N}) \simeq \THH(A \,|\, M)$, i.e. that the logarithmic $\THH$ only depends on the logarithmic structure.
\end{remark}
\appendix

\section{Relation to the Hopkins-Mahowald result}\label{Thom}

Theorem \ref{thm:free} about $\F_p \otimes_{\bS} \F_p$ is closely related to the following statement due to Hopkins and Mahowald. 
We thank Mike Mandell for explaining a proof to us. 
\begin{thm}[Hopkins, Mahowald]
\label{thm:hopmah}
The Thom spectrum of the $\E_2$-map
\[
\Omega^2 S^3 \to \BGL_1(\bS_p)
\]
corresponding to the element $1-p \in \pi_0(\GL_1(\bS_p))$ is equivalent to $\F_p$.
\end{thm}

We claim that this result is equivalent to Theorem \ref{thm:free}. More precisely we will show that each of the two results can be deduced from the other only using formal considerations and elementary connectivity arguments. 

\begin{lem}\label{equiv}
Theorem \ref{thm:hopmah} is equivalent to Theorem \ref{thm:free}.
\end{lem}
\begin{proof}
Let us first phrase Theorem \ref{thm:hopmah} a bit more conceptually following 
\cite{antolin2019simple}.
We can view $\Omega^2 S^3 \to \BGL_1(\bS_p)$ as the free $\E_2$-monoid on 
\[
S^1 \xto{1-p} \BGL_1(\bS_p)
\]
in the category $(\cS_{*})_{/\BGL_1(\bS_p)}$ of pointed spaces over $\BGL_1(\bS_p)$. The Thom spectrum functor $\cS_{/\BGL_1(\bS_p)} \to \Mod_{\bS_p}$ is symmetric-monoidal and thus the Thom spectrum of $\Omega^2 S^3$ can equivalently be described as the free $\E_2$-algebra over $\bS_p$ on the pointed $\bS_p$-module obtained as the Thom spectrum of $S^1 \xto{1-p} \BGL_1(\bS_p)$. This is easily seen to be $\bS_p \to \bS_p/p$. Since the free $\E_2$-$\bS$-algebra on the pointed $\bS$-module $\bS\to \bS/p$ is already $p$-complete, it also agrees with this Thom spectrum. We will write this as $\Free^{\E_2}(\bS\to \bS/p)$. 
There is a map $\bS/p\to \F_p$ of pointed $\bS$-modules which induces an isomorphism on $\pi_0$. We get an induced map
\begin{equation}\label{Hopkins-Mahowald}
\Free^{\E_2}(\bS\to \bS/p)\to \F_p.
\end{equation}
Theorem \ref{thm:hopmah} is now equivalently phrased as the statement that the map \eqref{Hopkins-Mahowald} is an equivalence. 
Since both sides are $p$-complete, this is equivalent to the claim that the map is an equivalence after tensoring with $\F_p$. This is the map
\[
\Free^{\E_2}_{\F_p}(\F_p\to \F_p\otimes \bS/p) \to \F_p\otimes_{\bS} \F_p
\]
induced by the map $\F_p\otimes \bS/p\to \F_p\otimes_{\bS} \F_p$ of pointed $\F_p$-modules. It follows by elementary connectivity arguments that this map is an isomorphism on $\pi_0$ and $\pi_1$.

Now we have an equivalence $\F_p \otimes \bS/p \simeq \F_p \oplus \Sigma \F_p$ as pointed $\F_p$-modules. Thus, we can also write $\Free^{\E_2}_{\F_p}(\F_p\to \F_p\otimes \bS/p)$ as the free $\E_2$-algebra on the \emph{unpointed} $\F_p$-module $\Sigma \F_p$. Thus, the Hopkins-Mahowald result is seen to be equivalent to the claim that the map
\[
\Free^{\E_2}_{\F_p}(\Sigma \F_p) \to \F_p\otimes_{\bS} \F_p
\]
induced by a map $\Sigma \F_p\to \F_p\otimes_{\bS}\F_p$ which is an isomorphism on $\pi_1$, is an equivalence. But this is precisely Theorem \ref{thm:free}. 
\end{proof}

In Section \ref{sec_bokstedt} we have deduced B\"okstedt's theorem (Theorem \ref{thm:bok}) directly from Theorem \ref{thm:free}.  Blumberg--Cohen--Schlichtkrull deduce an additive version of B\"okstedt's theorem in \cite[Theorem 1.3]{MR2651551} from Theorem \ref{thm:hopmah}. A variant of this argument is also given in \cite[Section 9]{Blumberg}. We note that the argument that they use only works additively and does not give the ring structure on $\THH_*(\F_p)$. We will explain this argument now and also how to modify it to give the ring struture as well.

\begin{proof}[Proof of Theorem \ref{thm:bok} from Theorem \ref{thm:hopmah} ]
The Thom spectrum functor 
\[
\cS_{*/\BGL_1(\bS_p)}\to \Mod_{\bS_p}
\] preserves colimits and sends products to tensor products, and thus sends the unstable cyclic Bar construction of $\Omega^2 S^3$ to the cyclic Bar construction of $\F_p$. This identifies $\THH(\F_p)$ as an $\E_1$-ring with a Thom spectrum on the free loop space $LB\Omega^2 S^3 \simeq L\Omega S^3$. Now, using the natural fibre sequence of $\E_1$-monoids in $\cS_{/\BGL_1(\bS_p)}$, $\Omega^2 S^3 \to L\Omega S^3 \to \Omega S^3$, one can identify $\THH(\F_p)$ with $\F_p[\Omega S^3]$. For example, since this is a split fibre sequence of $\E_1$ monoids, one gets an equivalence $L\Omega S^3 \simeq \Omega^2 S^3 \times \Omega S^3$ and thus an identification of $\THH(\F_p)$ as a tensor product of the Thom spectrum on $\Omega^2 S^3$ (i.e. $\F_p$) and the Thom spectrum on $\Omega S^3$. Thus, a Thom isomorphism yields an equivalence $\THH(\F_p) \simeq \F_p[\Omega S^3]$. But the equivalence $L\Omega S^3 \simeq \Omega^2 S^3 \times \Omega S^3$ is not an $\E_1$-map, so this argument only describes $\THH(\F_p)$ additively.

One can fix this as follows. The Thom spectrum can be interpreted as the colimit of the functor $L\Omega S^3 \to \Sp$ obtained by postcomposing with the functor $\BGL_1(\bS_p) \to \Sp$ that sends the point to $\bS_p$. Instead of passing to the colimit directly, one can pass to the left Kan extension along the map $L\Omega S^3 \to \Omega S^3$. This yields a functor $\Omega S^3 \to \Sp$ which sends the basepoint of $\Omega S^3$ to the colimit along the fiber, i.e. the Thom spectrum over $\Omega^2 S^3$. But this is precisely $\F_p$. We thus obtain a functor $\Omega S^3 \to \BGL_1(\F_p)$, whose colimit is the Thom spectrum of $L\Omega S^3$. Since the original functor $L\Omega S^3 \to \Sp$ was lax monoidal, because it came from an $\E_1$ map, the Kan extension $\Omega S^3 \to \BGL_1(\F_p)$ is also an $\E_1$ map. But the space of $\E_1$ maps $\Omega S^3 \to \BGL_1(\F_p)$ agrees with the space of maps $S^2 \to \BGL_1(\F_p)$ and is thus trivial. So the resulting colimit $\THH(\F_p)$ is, as an $\E_1$ ring, given by $\F_p[\Omega S^3]$.
\end{proof}

We think that the proof of B\"okstedt's Theorem given in Section \ref{sec_bokstedt} directly from Theorem \ref{thm:free} is easier than the `Thom spectrum proof' presented in this section, since the latter first uses Theorem \ref{thm:free} to deduce the Hopkins-Mahowald theorem and then the (extended) Blumberg--Cohen--Schlichtkrull argument to deduce B\"okstedt's result. However, logically all three results (Theorems \ref{thm:bok}, \ref{thm:free} and \ref{thm:hopmah}) are equivalent as shown in Remark \ref{rem_equib} and Lemma \ref{equiv}. So either can be deduced from the others. It would be nice to have a proof of one of these that does not rely on computing the dual Steenrod algebra with its Dyer-Lashof operations (or dually the Steenrod algebra and the Nishida relations).

\bibliographystyle{amsalphaabrvd}
\bibliography{CyclotomicSpectra}

\end{document}